\documentclass[10pt]{article}
\usepackage[utf8]{inputenc}

\usepackage{psfrag}
\usepackage{color}
\usepackage{amsmath, amsfonts, amsthm}
\usepackage{graphicx}
\usepackage{amssymb}
\usepackage{rotating}
\usepackage{multirow}
\usepackage{relsize}

\usepackage{booktabs}
\usepackage{hyperref} 
\usepackage{arydshln}
\usepackage{multicol}
\usepackage{overpic}
\usepackage{stmaryrd}
\usepackage{pict2e}
\usepackage{pdfpages}
\usepackage{pgfplots}
    \usepgfplotslibrary{groupplots}
    \usetikzlibrary{spy,backgrounds}
    \pgfplotsset{compat=1.18}
    
\usepackage{todonotes}
\presetkeys{todonotes}{inline}{}

\newtheorem{theorem}{Theorem}
\newtheorem{remark}{Remark}
\newtheorem{definition}{Definition}
\newtheorem{lemma}{Lemma}

\DeclareMathOperator{\dive}{div}
\DeclareMathOperator{\grad}{\nabla}

\DeclareMathOperator{\pseudoinv}{\delta^\dagger}

\newcommand{\bunderline}[1]{\underline{#1\mkern-4mu}\mkern4mu }

\newcommand{\R}{\mathbb R}
\newcommand{\Hone}[1]{H^1(\Omega_{#1}) }
\newcommand{\normL}[2]{\|#1\|_{L^2(#2)}}

\newcommand{\normHzero}[2]{ \|#1\|_{H^{1/2}_{00}(#2) }}
\newcommand{\seminormH}[2]{|#1|_{H^1(\Omega_{#2})}}
\newcommand{\seminormHhalf}[2]{|#1|_{H^{1/2}(#2)}}
\newcommand{\seminormS}[2]{|#1|_{S'_{#2}}}
\newcommand{\jump}[1]{\llbracket #1 \rrbracket}

\newcommand{\sumN}{\sum_{i=0}^{N}}
\newcommand{\sumE}{\sum_{\mathcal E}}
\newcommand{\sumV}{\sum_{\mathcal V}}
\newcommand{\sumEij}{\sum_{j\ne i}}
\newcommand{\sumjNx}{\sum_{j \in \mathcal N_x}}

\newcommand{\vi}{v_i}
\newcommand{\vj}{v_j}
\newcommand{\vii}{v_{ij, i}}
\newcommand{\vij}{v_{ij, j}}
\newcommand{\uii}{u_{i,i}}
\newcommand{\uij}{u_{i,j}}
\newcommand{\uji}{u_{j,i}}
\newcommand{\ujj}{u_{j,j}}

\newcommand{\barEij}{\bunderline E_{ij}}
\newcommand{\barEji}{\bunderline E_{ji}}
\newcommand{\Eij}{E_{ij}}
\newcommand{\Eji}{E_{ji}}

\newcommand{\ThetaE}{\Theta_{\mathcal E}}
\newcommand{\ThetaV}{\Theta_{\mathcal V}}
\newcommand\numberthis{\addtocounter{equation}{1}\tag{\theequation}}

\makeatletter
	\let\@fnsymbol\@arabic
\makeatother

\title{Convergence analysis of BDDC preconditioners for composite DG discretizations of the cardiac cell-by-cell model}
\author{
        Ngoc Mai Monica Huynh\thanks{
	    Dipartimento di Matematica, Universit\`a degli Studi di Pavia, 
		Via Ferrata, 27100 Pavia, Italy.
		E-mail: {\sf ngocmaimonica.huynh@unipv.it},
                {\sf luca.pavarino@unipv.it}.
		}
	\and 
	Fatemeh Chegini\thanks{
	    Zuse Institute Berlin, 
        Takustraße 7, 14195 Berlin, Germany. 
	    E-mail: {\sf chegini@zib.de},
	            {\sf weiser@zib.de}.
	}
	\and
	Luca Franco Pavarino\footnotemark[1]
	\and 
	Martin Weiser\footnotemark[2]
	\and
	Simone Scacchi\thanks{
	    Dipartimento di Matematica, Universit\`a degli Studi di Milano,
		Via Saldini 50, 20133 Milano, Italy.
		E-mail: {\sf simone.scacchi@unimi.it}.
		}
	}
\date{}
\begin{document}

\maketitle

\begin{abstract}
	A Balancing Domain Decomposition by Constraints (BDDC) preconditioner is constructed and analyzed for the solution of composite Discontinuous Galerkin discretizations of reaction-diffusion systems of ordinary and partial differential equations arising in cardiac cell-by-cell models. Unlike classical Bidomain and Monodomain cardiac models, which rely on homogenized descriptions of cardiac tissue at the macroscopic level, the cell-by-cell models enable the representation of individual cardiac cells, cell aggregates, damaged tissues, and nonuniform distributions of ion channels on the cell membrane.  The resulting discrete cell-by-cell models exhibit discontinuous global solutions across the cell boundaries. Therefore, the proposed BDDC preconditioner employs appropriate dual and primal spaces with additional constraints to transfer information between cells (subdomains) without affecting the overall discontinuity of the global solution. A scalable convergence rate bound is proved for the resulting BDDC cell-by-cell preconditioned operator, while numerical tests validate this bound and investigate its dependence on the discretization parameters.
\end{abstract}

\section{Introduction}

The purpose of this work is to design and analyze a Balancing Domain Decomposition by Constraints (BDDC) preconditioner for the solution of composite Discontinuous Galerkin (DG) discretizations of cardiac cell-by-cell models. 
The latter have been introduced in recent years in order to overcome some of the limitations of the macroscopic Bidomain and Monodomain models, which are based on a homogenized description of the cardiac tissue at  the macroscopic level, see \cite{tung1978, veneroni2009reaction}, \cite[Ch.~3]{franzone2014mathematical}. In this description, cell membrane, intracellular and extracellular spaces coexist at each point of the myocardium, see Figure \ref{fig:homogenization}, and therefore these macroscopic cardiac models cannot be used for a spatial description at the cellular scale.
Moreover, the Bidomain model is based on the assumption that all cells are uniformly connected, which is not always valid.
	
	\begin{figure}
		\centering
		\includegraphics[scale=.35]{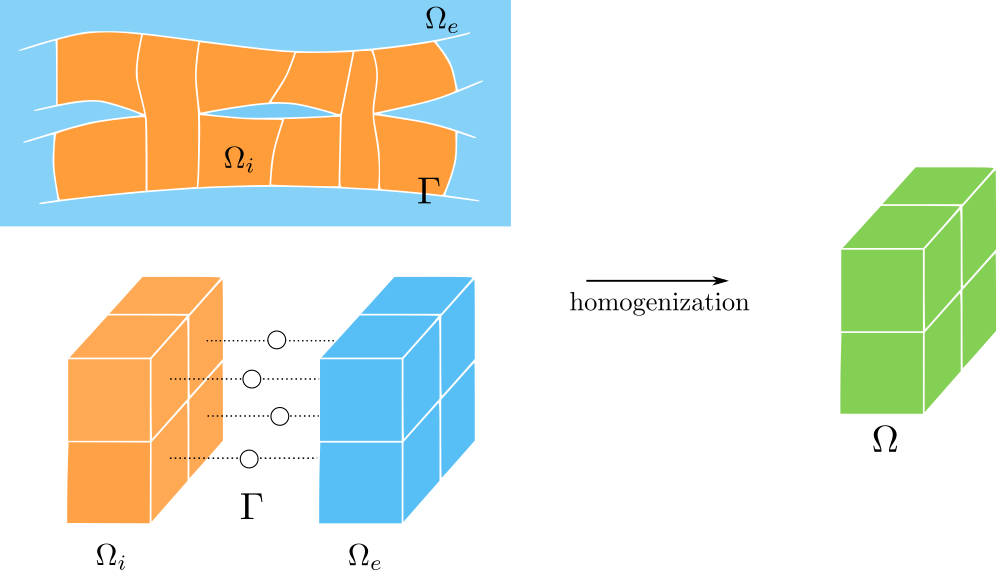}
		\caption{Exemplification of homogenization: the intracellular space $\Omega_i$ of a cell (depicted in orange, on the left side) is assumed to co-exist at every point of the myocardium $\Omega$ (representation on the right side, in green) with the extracellular space $\Omega_e$ (depicted in light blue) and the cell membrane $\Gamma$.}
		\label{fig:homogenization}
	\end{figure}

These limitations can be removed by considering cell-by-cell models based on representing individual cardiac cells, which are important for realistic modeling of damaged tissue and for studying the effects of nonuniform distributions of ion channels on the cell membrane. 		
Among the first cell-by-cell mathematical models, the EMI (Extracellular space, cell Membrane and Intracellular space) model has been derived and studied in \cite{tveito2021bis, tveito2021tris, jaeger2023reintroducing, tveito2017, tveito2021}.
This model considers these three domains in a coupled manner, allowing for discontinuous potentials across cell boundaries, general representations of cell geometries, as well as particular distributions of ion charges on the cell membranes. 
Other tridomain cardiac models have been proposed in \cite{bader1_2022,bader2_2022,sathar2015}, both at the microscopic and macroscopic level.
	
In this paper, we consider DG discretizations \cite{cockburn2012discontinuous} of cell-by-cell models in order to approximate the discontinuous nature of cellular aggregates and the associated electrical potentials. 
Arbitrary discretizations can be considered for each cell (or subdomain), allowing us to consider conforming standard finite element discretizations in each cell, while the global solution is allowed to be discontinuous across the cells. Other recent works focus on boundary elements methods, we refer to \cite{pezzuto2022, desouza2023boundary} and references therein. 
	
While efficient parallel preconditioners have been proposed for the standard Bidomain and Monodomain cardiac models, see \cite{plank2007algebraic}, \cite{vigmond2008solvers} and \cite[Ch.~8]{franzone2014mathematical}, the design of scalable preconditioners is still an open problem for the more challenging cardiac cell-by-cell models and only few works exploit the potentialities of modern computing architectures \cite{potse_siampp2022}.
Previous works on dual-primal preconditioners for continuous Galerkin discretizations of the Bidomain equations, \cite{zampini2014dual} for implicit-explicit time discretizations and \cite{huynh2022parallel,huynh2021newton} for implicit time discretizations, cannot be applied to cell-by-cell models due to the discontinuous Galerkin discretization required to approximate discontinuous potentials.
So far, few studies have focused on dual-primal preconditioners for DG discretizations of scalar elliptic partial differential equations \cite{canuto2014bddc, dryja2013analysis, dryja2015deluxe}.
By leveraging the previous works \cite{dryja2013analysis, dryja2015deluxe}, we design Balancing Domain Decomposition by Constraints (BDDC) preconditioners for composite DG-type discretizations of cardiac cell-by-cell/EMI models,
where each cell is associated with a subdomain of our domain decomposition. 	
One key point of our BDDC construction in a DG setting is the definition of extended dual and primal spaces for the degrees of freedom. Dual-primal methods in standard continuous Galerkin settings partition the degrees of freedom belonging to the subdomain boundaries into dual and primal, where the global solution is considered by the iterative algorithm as discontinuous and continuous, respectively; see \cite{dohrmann2003}, \cite[Ch. 6]{toselli2006domain}. The situation is different in our DG setting where the global solution fields of the cell-by-cell model are required to be discontinuous across subdomains.
Indeed, it is necessary to enrich dual and primal spaces with additional constraints which transfer information between subdomains without influencing the overall discontinuity of the global solution. We note that while the focus of this paper is on cardiac cell-by-cell/EMI models, our study can be extended to DG discretizations of general elliptic and parabolic equations.

The paper is structured as follows: in Section \ref{sec: model}, we briefly introduce a simplified cardiac cell-by-cell model, referring to \cite{tveito2021bis, tveito2017, tveito2021} for further details on the derivation of the EMI model; we present instead concise formulations for time and space composite-DG discretizations. Then in Section~\ref{sec: preconditioner}, we introduce the BDDC preconditioner for this type of composite-DG discretizations. The main result of the paper can be found in Section \ref{sec: theoretical convergence}, where a complete proof for the convergence is provided. Numerical tests validate the theoretical bound through serial simulations, reported in Section~\ref{sec: numerical results}. Section \ref{sec: conclusions} concludes the paper with synthetic discussion and comments.
	
	
	\section{The cell-by-cell model}\label{sec: model}
	In this paper, we consider a simplified EMI model, assuming that two or more cells can be connected neighbours through the cell membrane, without distinguishing between gap junctions and secondary ionic channels. A previous study on microscopic reaction-diffusion equations in cardiac electrophysiology can be found in \cite{veneroni2006reaction}.
	
	We consider $N$ cells immersed in the extracellular liquid, which altogether form the cardiac tissue $\Omega$, where $\Omega\subset\R^d$, with $d\in\{2,3\}$.
	We assume that each of these $N+1$ objects constitutes a single subdomain $\Omega_i$, with $i=0,\dots,N$, where we denote with $\Omega_0$ the extracellular subdomain.
	These cells interact with the extracellular surroundings and their neighboring cells by means of the ionic currents. We also assume that the intracellular spaces are connected through the gap junctions, special protein channels which allow the passage of ions directly between two intracellular environments, \cite{rohr2004}
	
	\begin{figure}[!ht]
		\centering
		\includegraphics[scale=.35]{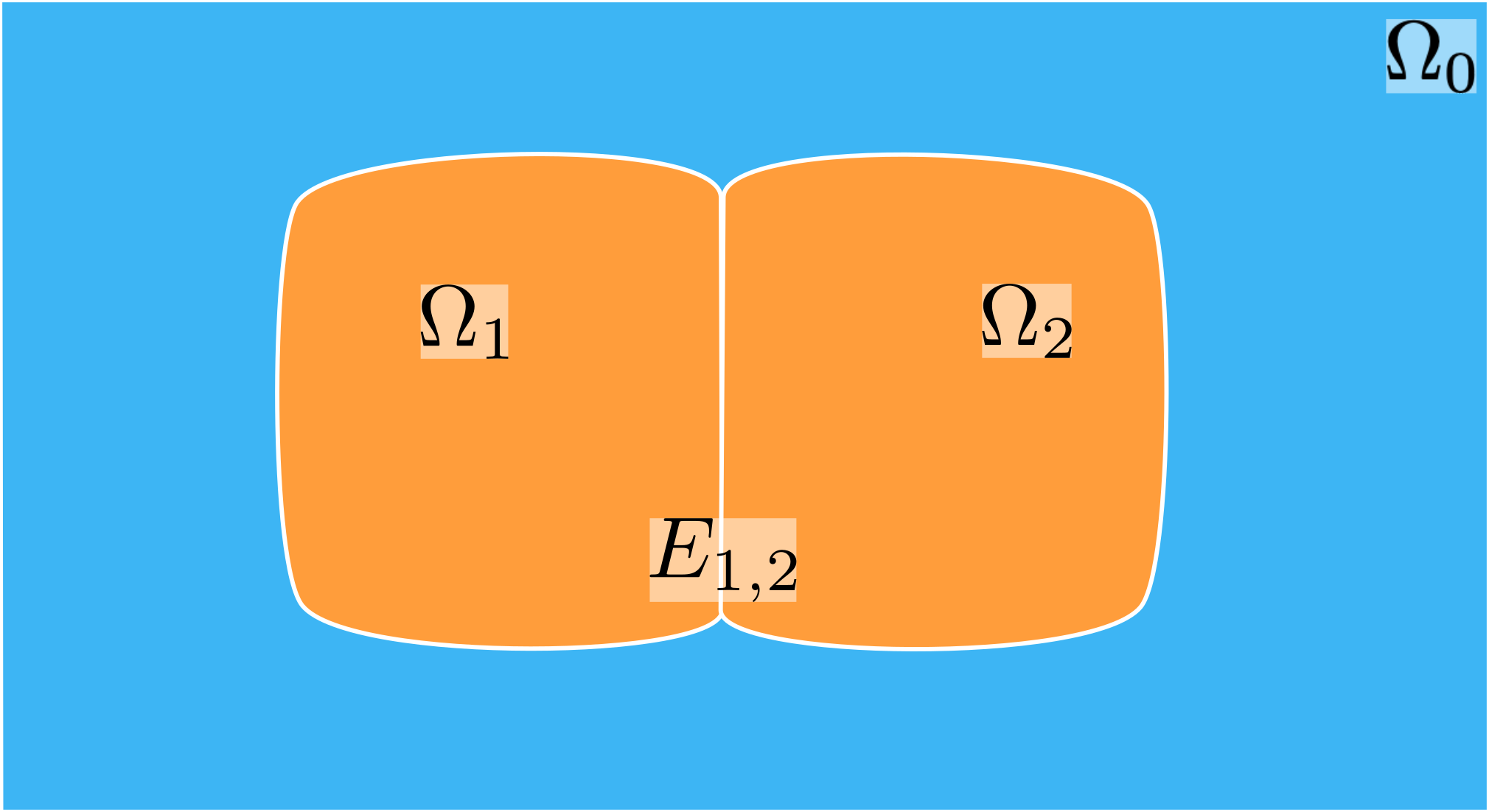}
		\caption{\emph{Simplified EMI model.} In this example, we depict a simple situation with only two cells, whose intracellular spaces are denoted by $\Omega_1$ and $\Omega_2$, while $\Omega_0$ denotes the extracellular surrounding; $F_{1,2}$ is assumed to be the intersection of the cell membrane between the two cells.  }
	\end{figure}
	
	This framework can be described by the following model
	\begin{equation} \label{eq: cell-by-cell model}
		\begin{cases}
			- \dive (\sigma_i \grad u_i) = 0		&\qquad \text{in }\Omega_i, \ \ i=0,\dots, N	\\
			- n_i^T \sigma_i \grad u_i = C_m \dfrac{\partial v_{ij} }{\partial t} + F(v_{ij})			&\qquad \text{on } \Eij = \overline{\Omega}_i \cap \overline{\Omega}_j \subset \partial \Omega_i, \; i\ne j		\\
            n^T\sigma_i\nabla u_i  = 0 & \qquad \text{on $\partial\Omega_i \cap \partial\Omega$} \\
			\dfrac{\partial c}{\partial t} - C(v_{ij}, w, c) = 0, 
			\qquad 
			\dfrac{\partial w}{\partial t} - R(v_{ij}, w) = 0,  
		\end{cases}
	\end{equation}	
	where $\sigma_i$ is the conductivity coefficient in $\Omega_i$, $n_i$ the outward normal on $\partial \Omega_i$, and $C_m$ the membrane capacitance for unit area of the membrane surface. The last two equations in the system (\ref{eq: cell-by-cell model}) model the ion flow dynamic by means of ordinary differential equations, describing the time evolution of ion concentrations $c$ and gating variables $w$.
	The transmembrane voltage $v_{ij} = u_i - u_j$ represents the jump in the value of the electric potentials between either two cells or a cell and the extracellular space, while $F(v_{ij})$ represents either the ionic current $I_\text{ion}(v_{ij}, c, w)$ or the gap junction current $G(v_{ij})$. 
	In this study, the gap junction terms $G(v)$ are assumed to be linear in the potential jumps $v$. We refer to \cite{tveito2021bis, tveito2017} for a formal derivation of the complete EMI system. 

 The solution obtained from system \eqref{eq: cell-by-cell model} is defined up to an arbitrary constant: in order to ensure uniqueness from a numerical point of view, we impose zero average over its domain for the extracellular component $u_0$. Moreover, we remark that this is not a standard parabolic problem but, as for the Bidomain model, it is a partial Differential Algebraic Equations (DAEs) system.
	
	We consider a splitting strategy for the time solution of system (\ref{eq: cell-by-cell model}). At each time step, we solve first the ionic model, given the jump $v_{ij}$ from the previous time step; then, we update the cell-by-cell model with the newly-computed $c$ and $w$ and solve it with respect to the electric potential. In this sense, throughout the next Sections, we will focus only on the discretization and preconditioning strategy of the cell-by-cell model.
	Additionally, since the purpose of this work is to study the theoretical convergence of a preconditioner for the resulting discrete linear system, we consider the simplified situation of a two-dimensional setting, leaving the three-dimensional case for further numerical studies.

	\subsection{Weak formulation}
	Integrating by parts over $\Omega_i$ the first equations in (\ref{eq: cell-by-cell model}) and using the second equation on the membrane, we obtain the  $i$-th problem: find $u_i \in  \Hone{i}$ such that 
	\begin{align*}
		0 &= - \int_{\Omega_i} \dive (\sigma_i \grad u_i) \, \phi_i \, dx		
		= \int_{\Omega_i} \sigma_i \grad u_i \grad \phi_i \, dx 	- 	\int_{\partial \Omega_i} n_i^T \sigma_i \grad u_i \phi_i \, ds	\\
		&= \int_{\Omega_i} \sigma_i \grad u_i \grad \phi_i \, dx 	+ \sumEij \int_{\Eij} \left( C_m \dfrac{\partial v_{ij} }{\partial t} + F(v_{ij}) \right) \phi_i \, ds
			\qquad
		\forall \phi_i \in \Hone{i}.
	\end{align*}
	By summing all contributions from the $N+1$ subdomains, the global problem is given by
	\begin{equation*}
		\sumN \int_{\Omega_i} \sigma_i \grad u_i \grad \phi_i \, dx 	+ \dfrac{1}{2} \sumN \sumEij \int_{\Eij} \left( C_m \dfrac{\partial \jump{u}_{ij} }{\partial t} + F(\jump{u}_{ij}) \right) \jump{\phi}_{ij} \, ds = 0,
	\end{equation*}
	where we denote with $\jump{u}_{ij} = u_i - u_j$ the jump between the value of the electric potential $u_i$ and its neighboring $u_j$ from the adjacent subdomain $\Omega_j$ along the common boundary $\Eij \subset \partial \Omega_i$.

	\subsection{Time discretization}
	We consider an implicit-explicit (IMEX) time discretization, where the diffusion term is treated implicitly and the reaction explicitly. We discretize the time interval $[0,T]$ into $K$ intervals and, by defining the time step $\tau = t^{k+1} - t^k$, for $k=0,\dots,K$, we obtain the following scheme: 
	\begin{equation*}
		\dfrac{1}{2} \sumN \sumEij \int_{\Eij} C_m \dfrac{\jump{u^{k+1}}_{ij} - \jump{u^k}_{ij} }{\tau} \jump{\phi}_{ij} \, ds \, + \sumN \int_{\Omega_i} \sigma_i \grad u_i^{k+1} \grad \phi_i \, dx 
		= -\dfrac{1}{2} \sumN \sumEij \int_{\Eij} F(\jump{u^k}_{ij}) \jump{\phi}_{ij} \, ds ,
	\end{equation*}
	and, by rearranging the terms, 
	\begin{multline*}
		\dfrac{1}{2} \sumN \sumEij \int_{\Eij} C_m \jump{u^{k+1}}_{ij} \jump{\phi}_{ij} \, ds + \tau \sumN \int_{\Omega_i} \sigma_i \grad u_i^{k+1} \grad \phi_i \, dx \\
		= \dfrac{1}{2} \sumN \sumEij \int_{\Eij} C_m \jump{u^{k}}_{ij} \jump{\phi}_{ij} \, ds - \tau \dfrac{1}{2} \sumN \sumEij \int_{\Eij} F(\jump{u^k}_{ij}) \jump{\phi}_{ij} \, ds.
	\end{multline*}
	We define the above quantities as follows:
	\begin{equation} \label{eq: a_i p_i}
		\begin{split}
			a_i (u_i, \phi_i) &:= \int_{\Omega_i} \sigma_i \grad u_i \grad \phi_i \, dx					\\
		p_i (u_i, \phi_i) &:= \dfrac{1}{2} \sumEij \int_{\Eij} C_m \jump{u}_{ij} \jump{\phi}_{ij} \, ds	\\
		f_i (\phi_i) &:= \dfrac{1}{2} \sumEij \int_{\Eij} \left( C_m \jump{u}_{ij} \jump{\phi}_{ij} - \tau F(\jump{u}_{ij}) \jump{\phi}_{ij} \right) \, ds 
		\end{split}
	\end{equation}
	\begin{equation} \label{eq: widehat a_i}
        d_i (u_i, \phi_i) := \tau a_i(u_i, \phi_i) + p_i(u_i, \phi_i).
	\end{equation}

	\subsection{Space discretization}
	n the rest of the paper, we will focus only on the two-dimensional case ($d=2$).
    Let $V_i(\overline \Omega_i)$ be the regular finite element space of piecewise continuous linear functions in $\overline \Omega_i$ and define \mbox{$V (\Omega) = V_0(\overline \Omega_0) \times \dots \times V_N(\overline \Omega_{N})$} as the global finite element space. 

	We follow the notation proposed in \cite{dryja2015deluxe}:
	\begin{itemize}
	    \item We denote with $\barEij := \partial \Omega_i \cap \partial \Omega_j$ and $\barEji := \partial \Omega_j \cap \partial \Omega_i$ an edge of $\partial \Omega_i$ and $\partial \Omega_j$ respectively; we observe that, although $\barEij$ and $\barEji$ are geometrically the same object, they are treated separately, since one could possibly consider different triangulations on $\barEij$ and $\barEji$, e.g., \cite{dryja2013analysis, dryja2015deluxe}.
	    \item We define $\mathcal E_i^0$ the set of indices $j$ of subdomain $\Omega_j$ that share an edge $\Eji$ with $\Omega_i$. Analogously, we introduce $\mathcal E_i^\partial$ as the set of indices referred to the edges of $\Omega_i$ which belong to $\partial \Omega$. Therefore, the set of indices of all edges of $\Omega_i$ is denoted by
	    $$
	        \mathcal E_i := \mathcal E_i^0 \cup \mathcal E_i^\partial.
	    $$
	\end{itemize}
	The general problem reads: find $u=\{u_i\}_{i=0}^{N} \in V(\Omega)$ such that 
	\begin{equation} \label{eq: glob sys}
		d_h(u, \phi) = f(\phi),
		\qquad \forall\phi=\{\phi_i\}_{i=0}^{N} \in V(\Omega),
	\end{equation}
where
$		d_h (u, \phi) := \sumN d_i (u_i, \phi_i) = \sumN \left( \tau \, a_i (u_i, \phi_i) + p_i (u_i, \phi_i)\right), \ \ 
		f(\phi) := \sumN f_i (\phi_i). $
Denoting by$A_i$ and mass $M_i$ the local stiffness and mass matrices, (\ref{eq: glob sys}) can be written 
in matrix form as
	\begin{equation} \label{eq: global algebraic system}
	    \mathcal K \mathbf u = \mathbf f,
	    \qquad
	    \text{with } 
	    \qquad \mathcal K = \sum_{i=0}^N \mathcal K_i,
	    \qquad
	    \mathcal K_i = \tau A_i + M_i.
	\end{equation}
	
	\section{BDDC, spaces, scaling} \label{sec: preconditioner}
	In this Section, we will construct suitable domain decomposition dual-primal spaces for composite-DG discretizations of the cell-by-cell model, with minimal graphic and matrix examples. Then, we will define the BDDC preconditioner on these enriched spaces and we conclude with some technical tools that will be employed in the next Section for the theoretical proof. The construction of the BDDC solver is inspired by the previous works \cite{dryja2013analysis, dryja2015deluxe} on dual-primal solvers for DG discretizations of scalar elliptic equations.
	
	\subsection{ Non-overlapping dual-primal spaces} 
	Let us assume that the $N+1$ subdomains form a partition of the cardiac tissue $\Omega$, such that $\overline{\Omega} = \cup_{i=0}^{N} \overline{\Omega}_i$, $ \Omega_i \cap \Omega_j = \emptyset $ if $ i \neq j$ and the intersection of the boundaries $ \partial \Omega_i \cap \partial \Omega_j $ is either empty, a vertex or a face. 
	Each subdomain is a union of shape-regular conforming finite elements.
	The interface $\Gamma$ is the set of points that belong to at least two subdomains,
	\begin{equation*}
		\Gamma := \bigcup_{i=0}^N \Gamma_i, 	\qquad \qquad 	\Gamma_i: = \overline{\partial \Omega_i \backslash \partial \Omega}.
	\end{equation*} 
	In order to correctly take into account the natural jumps of the electric potentials along the boundaries of each cell, we deal with DG-type discretizations, where each of the degrees of freedom belonging to the interface will have multiplicity depending on the number of subdomains (or cells) that shares it.
	
	We assume that subdomains are shape regular and have a characteristic diameter of size $H_i$ whereas the finite elements are of diameter $h$; we denote by $H = \max_i H_i$. 

	Denoting $\Omega_i' = \overline \Omega_i \cup  \bigcup_{j \in \mathcal E_i^0} \barEji$,  the union of nodes in $\overline \Omega_i$ and nodes on edges $\barEji \subset \partial \Omega_j$, with $j \in \mathcal E_i^0$, we define the associated local finite element spaces 
	$$
	    W_i(\Omega_i') = V_i (\overline \Omega_i) \times \prod_{j \in \mathcal E_i^0} W_i (\barEji),
	$$
	with $W_i (\barEji)$ the trace of the space $V_j (\overline \Omega_j)$ on $\barEji \subset \partial \Omega_j$, for all $j \in \mathcal E_i^0$. A very schematic representation can be found in Figure \ref{fig: subdomain dofs}, where for simplicity we have considered the contribution of four adjacent subdomains $\Omega_j$; moreover, let $\Xi_\Omega$ denote the index set of the geometric object $\Omega$.
        \begin{figure}
	    \centering
	    \includegraphics[scale=.25]{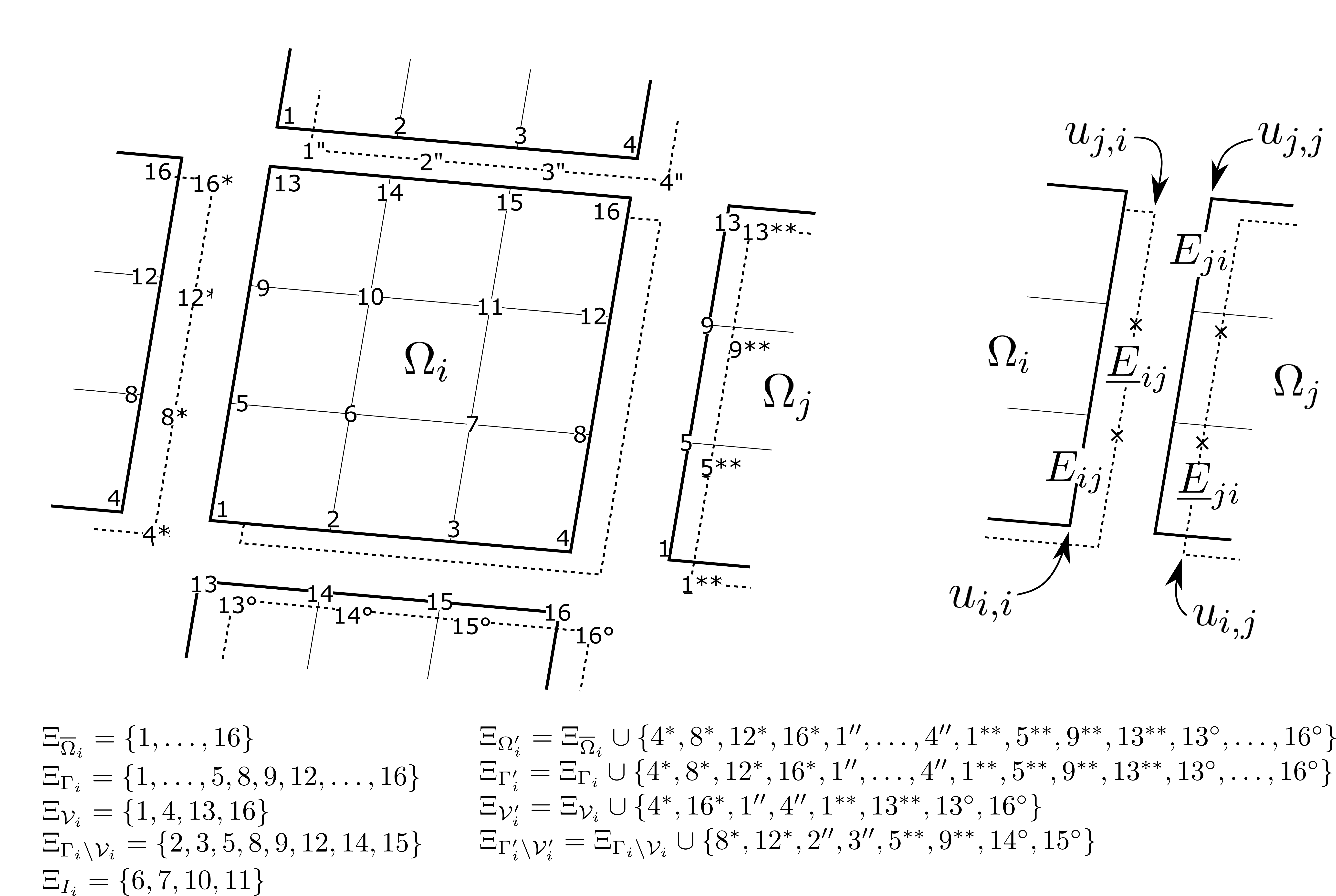}
	    \caption{Schematic partition of degrees of freedom of the substructure $\Omega_i$ (left) and zoom over the common geometric boundary $\partial \Omega_i \cap \partial \Omega_j$ (right). The index sets below the Figures are related to the substructure $\Omega_i$.}
	    \label{fig: subdomain dofs}
	\end{figure}
	Each element of this space can be represented as
	$ u_i = (\uii, \{ \uij \}_{j \in \mathcal E_i^0} ),$
	see Figure~\ref{fig: subdomain dofs} on the right.
	For example, in a two-dimensional structured decomposition, where each subdomain $\Omega_i$ can be surrounded by four subdomains $\Omega_{j_1}$, $\Omega_{j_2}$, $\Omega_{j_3}$, $\Omega_{j_4}$, we have
	\begin{equation} \label{eq: u_i}
	    u_i = (\uii, \, u_{i,j_1}, \, u_{i,j_2}, \, u_{i,j_3}, \, u_{i,j_4} ).
	\end{equation}
	We note that this representation naturally induces the definition of $d_i$ on the spaces $W_i(\Omega_i') \times W_i(\Omega_i')$ from Eq. (\ref{eq: widehat a_i}), along with its corresponding matrix $\mathcal K'_i$. 
	
	The space $W_i(\Omega_i')$ is further partitioned into its interior part $W_i(I_i)$ and into the finite element trace space $W_i(\Gamma_i')$, 
	$$
	    W_i(\Omega_i') = W_i(I_i) \times W_i(\Gamma_i'),
	$$
	where $\Gamma_i' := \Gamma_i \cup \left\{ \bigcup_{j \in \mathcal E_i^0} \barEji \right\}$ denote the local interface nodes in $\Omega_i'$.
	This representation allows us to define (\ref{eq: u_i}) also as
	$
	    u_i = ( u_{i,I}, \, u_{i, \Gamma'} )
	$,
	where $u_{i,I}$ denotes the values of $u_i$ at the interior nodal points on $I_i$, while $u_{i, \Gamma'}$ are the values at the nodal points on $\Gamma_i'$.
	We consider the product spaces 
	$$
		W(\Omega') := \prod_{i=0}^{N} W_i (\Omega_i'),
	\qquad
		W (\Gamma') :=  \prod_{i=0}^{N} W_i(\Gamma_i'),
	$$
	where $u \in W(\Omega')$ means that $u=\{u_i\}_{i=0}^{N}$, with $u_i \in W_i (\Omega_i')$ and, in a very similar way, $u_{\Gamma'} \in W(\Gamma')$. Here $\Gamma'=\prod_{i=0}^N \Gamma_i'$ is the global broken interface. The latter space can be interpreted as the subspace of $W(\Omega')$ of functions which are discrete harmonic in the sense of $\mathcal H'_i$ in each $\Omega_i'$.
	On a general side, for our purposes, we define the discrete harmonic extension of the Laplacian operator on $\partial \Omega_i$ as follows, and we refer to \cite{dryja2013analysis} for more details. 
	
	\begin{definition}\label{def:harmonic-extension}
		We define $\mathcal{H}_i^\Delta$ the discrete harmonic extension of the Laplacian operator on $\partial \Omega_i$ as
		\begin{equation*}
			u = \mathcal{H}_i^\Delta u_\Gamma 
			\qquad
			\Leftrightarrow
			\qquad
			\begin{cases}
				- \Delta u = 0		&\text{in } \partial \Omega_i	\\
				u = u_\Gamma 		&\text{on }	\Gamma_i		\\
				u = 0				&\text{on }	\partial \Omega_i \backslash \Gamma_i
			\end{cases}
		\end{equation*}
	\end{definition}
	
	In dual-primal methods, we iterate in the space $W$ while requiring \textit{primal constraints} (which are continuity constraints) to hold throughout the iterations. These constraints guarantee that each subdomain problem is invertible and that a good convergence bound can be obtained.
	Let us denote with $\mathcal V_i$ and $\mathcal V_i'$ the set of nodal points associated to the corner unknowns of $\overline \Omega_i$ and $\Omega_i'$ respectively
	\begin{equation*}
	    \mathcal V_i := \bigcup_{j \in \mathcal E_i^0} \partial \Eij  ,
	    \qquad
	    \mathcal V_i' := \mathcal V_i \cup \bigcup_{j \in \mathcal E_i^0} \partial \Eji  .
	\end{equation*}
	
	As in \cite[Definition 3.3]{dryja2015deluxe}, we consider $\widetilde W(\Omega')$ as the space of finite element functions in $W(\Omega')$ which are continuous on all $\mathcal V_i'$ in the following sense.
	\begin{definition}
	    A function $u=\{u_i\}_{i=0}^{N} \in W(\Omega')$ is continuous at the corners $\mathcal V_i'$, if  
	    \begin{align*}
	        u_{i, \,\partial \Omega_i} (x) &= u_{j, \,\Eij} (x)       
	        \qquad
	        \text{at } x \in \partial \Eij, 
	        \qquad
	        \text{for all } j \in \mathcal E_i^0 ,   \\
	        u_{i, \, \Eji} (x) &= u_{j, \, \partial \Omega_j} (x)       
	        \qquad
	        \text{at } x \in \partial \Eji, 
	        \qquad
	        \text{for all } j \in \mathcal E_i^0   ,
	    \end{align*}
	    being $u_{i, \, \star}$ the values of function $u_i$ on $\star$
	    for $i=0, \dots, N$, $\star \in \{ \partial \Omega_i, \, \partial \Omega_j, \, \Eij, \,  \Eji \}$.
	    The subspace of $W(\Omega')$ of continuous functions at the corners $\mathcal V_i'$, for $i=0, \dots, N$, is denoted by $\widetilde W(\Omega')$. Conversely, $\widetilde W(\Gamma')$ is the subspace of $\widetilde W(\Omega')$ of functions which are discrete harmonic in the sense of $\mathcal H_i'$. Clearly, we have
	    \begin{equation*}
	        \widetilde W(\Gamma') \subset \widetilde W(\Omega') \subset W(\Omega').
	    \end{equation*}
	\end{definition}
	In this sense, any function $u \in \widetilde W(\Omega')$ can be represented as
 $	  u = (u_I, \, u_\Delta, \, u_\Pi),$
	where the subscript $I$ denotes the \emph{interior} degrees of freedom at nodal points $I=\prod_{i=0}^N I_i$, $\Pi$ refers to the corners $\mathcal V_i'$ (also called \emph{primal}), and $\Delta$ refers to the remaining nodal points $\Gamma_i' \backslash \mathcal V_i'$ (also known as \emph{dual}), for all $i=0,\dots, N$. 
    We introduce the spaces
    $$
        W_\Delta (\Gamma') = \prod_{i=0}^N W_{i, \Delta} (\Gamma_i')
    \qquad
    \text{and}
    \qquad
        \widetilde W_\Pi (\Gamma'),
    $$
    where $W_{i, \Delta} (\Gamma_i')$ are the local spaces referred to the degrees of freedom associated to the nodes of $\Gamma_i' \backslash \mathcal V_i'$.
    Thus, we can further decompose the space $\widetilde W(\Gamma')$ as
    \begin{equation*}
        \widetilde W(\Gamma') = W_\Delta (\Gamma') \times \widetilde W_\Pi (\Gamma'),
    \end{equation*}
    from which it follows the representation
    \begin{equation*}
        u_{\Gamma'} \in \widetilde W(\Gamma') 
        \qquad
        u_{\Gamma'}=(u_\Delta, \, u_\Pi),
    \end{equation*}
    with $u_\Pi \in \widetilde W_\Pi (\Gamma')$ and $u_\Delta = \{ u_{i, \Delta} \}_{i=0}^N \in W_\Delta (\Gamma')$. It is important to notice that $u_{i, \Delta}$ can be partitioned as
    \begin{equation*}
        u_{i, \Delta} = \{ u_{i, \Eij}, \, u_{i, \Eji} \}_{j \in \mathcal E_i^0},
    \end{equation*}
    where $u_{i, \Eij}$ and $u_{i, \Eji}$ are the restrictions of $u_{i, \Delta}$ to $\Eij$ and $\Eji$ respectively.
    
    Lastly, we introduce the space 
    \begin{equation*}
        \widehat W_\Delta (\Gamma) = \prod_{i=0}^N V_i (\Gamma_i \backslash \mathcal V_i),
    \end{equation*}
    where $V_i (\Gamma_i \backslash \mathcal V_i)$ is the restriction of $V_i (\Gamma_i)$ to $\Gamma_i \backslash \mathcal V_i$. 
	
	The reordering of the degrees of freedom induces a similar reordering at the algebraic level: assuming that the local problems (\ref{eq: widehat a_i}) can be written in algebraic form as $\mathcal K'_i \, u_i = f_i$, we can write
	\begin{equation*}
		\mathcal K'_i = 
		\begin{bmatrix}
			K'_{i, \, II}			    & K'_{i, \, I\Gamma'} \\
			K'_{i, \, \Gamma' I} 	    & K'_{i, \, \Gamma' \Gamma'}
		\end{bmatrix}.
	\end{equation*}
	We observe that the matrix $\mathcal K$ from (\ref{eq: global algebraic system}), obtained by assembling the matrices $\mathcal K'_i$, for $i=0,\dots,N$, from $W(\Omega')$ to $\widetilde W(\Omega')$, is not block diagonal, due to the couplings between variables at the corners $\mathcal V_i'$.
	
	A very simple situation with only three subdomains (one extracellular plus two cells) is depicted in Figure~\ref{fig: enumeration for matrix}. Denote with $\Omega_0$ the extracellular domain and with $\Omega_1$, $\Omega_2$ the intracellular domains (blue, green and red colors respectively). 
	We consider the nodes on $\partial \Omega$ as interior to $\Omega_0$, since we will consider Neumann boundary conditions on $\partial \Omega$.
	The local enumeration of degrees of freedom easily reads as
	\begin{align*}
	    \Xi_{I_0}         &= \{ 1,\dots,13, 21, \dots,24, 27,\dots,30, 33,\dots,36, 44,\dots,56 \} \\
	    \Xi_{\Gamma_0}    &= \{ 14,\dots, 20,25, 26, 31, 32, 37, \dots, 43 \} \\
	    \Xi_{I_1}         &= \{ 62, 63, 66, 67 \} \\
	    \Xi_{\Gamma_1}    &= \{ 57,\dots,61, 64, 65, 68,\dots,72 \} \\
	    \Xi_{I_2}         &= \{ 78, 79, 82, 83 \} \\
	    \Xi_{\Gamma_2}    &= \{ 73,\dots,77, 80, 81, 84,\dots,88 \} 
	\end{align*}
	where $\Xi_{I_i}$ and $\Xi_{\Gamma_i}$ denotes the set of the indices of degrees of freedom interior and on the interface of subdomain $\Omega_i$ respectively.
	
	\begin{figure}[!ht]
	    \centering
	    \includegraphics[scale=.35]{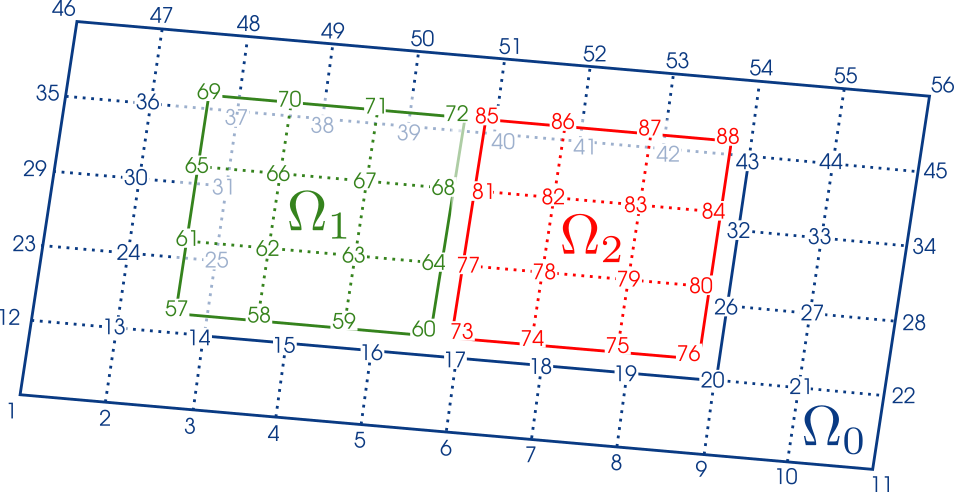}
	    \caption{Example of global enumeration of degrees of freedom for DG-type discretizations, in a non-overlapping domain decomposition setting. Simple case of three subdomains: extracellular $\Omega_0$ in blue, intracellular $\Omega_1$ and $\Omega_2$ in green and red.}
	    \label{fig: enumeration for matrix}
	\end{figure}
	In this example (we suppose $\tau = 1$ to simplify the notation) the matrix $\mathcal{K}$ is given by
	\begin{equation*}
	    \renewcommand\arraystretch{1.2}
    	\setlength{\dashlinedash}{0.9pt}
        \setlength{\dashlinegap}{4.5pt}
	    \mathcal{K} = 
	    \left[
	    \begin{array}{cc:cc:cc}
	        A_{I_0 I_0}         & A_{I_0 \Gamma_0}                                  & & & & \\
	        A_{\Gamma_0 I_0}    & A_{\Gamma_0 \Gamma_0} + M_{\Gamma_0 \Gamma_0}     & & -M_{\Gamma_0 \Gamma_1}      & & -M_{\Gamma_0 \Gamma_2} \\
	        \hdashline
	                            &                                                   & A_{I_1 I_1}       & A_{I_1 \Gamma_1}  & & \\
	                            & -M_{\Gamma_1 \Gamma_0}                            & A_{\Gamma_1 I_1}  & A_{\Gamma_1 \Gamma_1} + M_{\Gamma_1 \Gamma_1}   & & -M_{\Gamma_1 \Gamma_2}  \\
	       \hdashline
	                            & & &       & A_{I_2 I_2}         & A_{I_2 \Gamma_2}    \\
	                            & -M_{\Gamma_2 \Gamma_0}    &   & -M_{\Gamma_2 \Gamma_1}                            & A_{\Gamma_2 I_2}  & A_{\Gamma_2 \Gamma_2} + M_{\Gamma_2 \Gamma_2}
	    \end{array}
	    \right],
	\end{equation*}
	where we have highlighted the partition of degrees of freedom, thus resulting in a reordering of the entries of the local stiffness and mass matrices. We underline that the contribution of the mass term enters into account only on the interfaces, see Eq. (\ref{eq: a_i p_i}). Moreover, the mass matrix 
    $
        M_{\Gamma_0 \Gamma_0} 
    $
    is given by the sum of two contributions $M_{\Gamma_0 \Gamma_0}^{01}$ and $M_{\Gamma_0 \Gamma_0}^{02}$, obtained from the integration on the common boundary between $\Omega_0$ and $\Omega_1$ and between $\Omega_0$ and $\Omega_2$, respectively. Thus, 
    \begin{equation*}
        M_{\Gamma_0 \Gamma_0} = M_{\Gamma_0 \Gamma_0}^{01} + M_{\Gamma_0 \Gamma_0}^{02}.
    \end{equation*}
    Analogously, 
    \begin{equation*}
        M_{\Gamma_1 \Gamma_1} = M_{\Gamma_1 \Gamma_1}^{10} + M_{\Gamma_1 \Gamma_1}^{12}
        \qquad
        \text{and}
        \qquad
        M_{\Gamma_2 \Gamma_2} = M_{\Gamma_2 \Gamma_2}^{20} + M_{\Gamma_2 \Gamma_2}^{21}.
    \end{equation*}
    The local contributions instead, are given by
	\begin{equation*}
	    \renewcommand\arraystretch{1.2}
    	\setlength{\dashlinedash}{0.9pt}
        \setlength{\dashlinegap}{4.5pt}
	    K'_0 = 
	    \left[
	    \begin{array}{cc:cr:cr}
	        A_{I_0 I_0}         & A_{I_0 \Gamma_0}                                  & & & & \\
	        A_{\Gamma_0 I_0}    & A_{\Gamma_0 \Gamma_0} + \frac{1}{2} M_{\Gamma_0 \Gamma_0}     & & -\frac{1}{2} M_{\Gamma_0 \Gamma_1}      & & -\frac{1}{2} M_{\Gamma_0 \Gamma_2} \\
	        \hdashline
	                            &      & \color{white} A_{I_1 I_1}       & \color{white} A_{I_1 \Gamma_1}  & & \\
	                            & -\frac{1}{2} M_{\Gamma_1 \Gamma_0}                            & \color{white} A_{\Gamma_1 I_1}  & {\color{white} A_{\Gamma_1 \Gamma_1} + } \frac{1}{2} M_{\Gamma_1 \Gamma_1}^{10}   & & \color{white} -M_{\Gamma_1 \Gamma_2}  \\
	       \hdashline
	                            & & &       & \color{white} A_{I_2 I_2}         & \color{white} A_{I_2 \Gamma_2}    \\
	                            & -\frac{1}{2} M_{\Gamma_2 \Gamma_0}    &   & \color{white} -M_{\Gamma_1 \Gamma_2}                            & \color{white} A_{\Gamma_2 I_2}  & {\color{white} A_{\Gamma_2 \Gamma_2} + }\frac{1}{2} M_{\Gamma_2 \Gamma_2}^{20}
	    \end{array}
	    \right],
	\end{equation*}
	\begin{equation*}
	    \renewcommand\arraystretch{1.2}
    	\setlength{\dashlinedash}{0.9pt}
        \setlength{\dashlinegap}{4.5pt}
	    K'_1 = 
	    \left[
	    \begin{array}{cr:cc:cr}
	         \color{white}A_{I_0 I_0}         &  \color{white}A_{I_0 \Gamma_0}                            & & & & \\
	         \color{white}A_{\Gamma_0 I_0}    & { \color{white}A_{\Gamma_0 \Gamma_0} + } \frac{1}{2} M_{\Gamma_0 \Gamma_0}^{01}     & & -\frac{1}{2} M_{\Gamma_0 \Gamma_1}      & & \color{white} -M_{\Gamma_0 \Gamma_2} \\
	        \hdashline
	                            &                                                   & A_{I_1 I_1}       & A_{I_1 \Gamma_1}  & & \\
	                            & -\frac{1}{2} M_{\Gamma_1 \Gamma_0}                            & A_{\Gamma_1 I_1}  & A_{\Gamma_1 \Gamma_1} + \frac{1}{2} M_{\Gamma_1 \Gamma_1}   & & -\frac{1}{2} M_{\Gamma_1 \Gamma_2}  \\
	       \hdashline
	                            & & &       & \color{white} A_{I_2 I_2}         & \color{white} A_{I_2 \Gamma_2}    \\
	                            & \color{white} -M_{\Gamma_2 \Gamma_0}    &   & -\frac{1}{2} M_{\Gamma_2 \Gamma_1}                            & \color{white} A_{\Gamma_2 I_2}  & { \color{white} A_{\Gamma_2 \Gamma_2} +} \frac{1}{2} M_{\Gamma_2 \Gamma_2}^{21}
	    \end{array}
	    \right],
	\end{equation*}
	\begin{equation*}
	    \renewcommand\arraystretch{1.2}
    	\setlength{\dashlinedash}{0.9pt}
        \setlength{\dashlinegap}{4.5pt}
	    K'_2 = 
	    \left[
	    \begin{array}{cr:cr:cc}
	         \color{white}A_{I_0 I_0}         &  \color{white}A_{I_0 \Gamma_0}                                  & & & & \\
	         \color{white}A_{\Gamma_0 I_0}    & { \color{white}A_{\Gamma_0 \Gamma_0} + } \frac{1}{2} M_{\Gamma_0 \Gamma_0}^{02}     & &       & & -\frac{1}{2} M_{\Gamma_0 \Gamma_2} \\
	        \hdashline
	                            &                                                   & \color{white} A_{I_1 I_1}       & \color{white} A_{I_1 \Gamma_1}  & & \\
	                            & \color{white} -M_{\Gamma_1 \Gamma_0}                            & \color{white} A_{\Gamma_1 I_1}  & { \color{white}A_{\Gamma_1 \Gamma_1} + } \frac{1}{2} M_{\Gamma_1 \Gamma_1}^{12}   & & -\frac{1}{2} M_{\Gamma_1 \Gamma_2}  \\
	       \hdashline
	                            & & &       & A_{I_2 I_2}         & A_{I_2 \Gamma_2}    \\
	                            & -\frac{1}{2} M_{\Gamma_2 \Gamma_0}    &   & -\frac{1}{2} M_{\Gamma_2 \Gamma_1}                            & A_{\Gamma_2 I_2}  & A_{\Gamma_2 \Gamma_2} + \frac{1}{2} M_{\Gamma_2 \Gamma_2}
	    \end{array}
	    \right].
	\end{equation*}
	The {\it static condensation} procedure, which consists in eliminating the interior degrees of freedom, reduces the problem to one on the interface $\Gamma'$, where the local Schur complement systems are given by
	\begin{equation*}
		S'_i := K'_{i, \,\Gamma' \Gamma'} - K'_{i, \, \Gamma I} (K'_{i, \, II})^{-1}  	 K'_{i, \, I\Gamma'}.
	\end{equation*}
	By defining the unassembled Schur complement matrix $S' = \text{diag } \left[ S'_0, \dots, S'_N \right]$, we obtain the global Schur complement matrix $\widehat{S}_{\Gamma'} = R_{\Gamma'}^T S' R_{\Gamma'}$, where $R_{\Gamma'}$ is the direct sum of local restriction operators $R_{\Gamma'}^{(i)}$ returning the local interface components,
    $ R_{\Gamma'}^{(i)}: W(\Omega') \rightarrow W_i(\Omega_i')$.
	Thus, instead of solving system (\ref{eq: glob sys}), we solve the Schur complement system 
	\begin{equation}\label{schursys}
		\widehat{S}_{\Gamma'} u_{\Gamma'} = \widehat{f}_{\Gamma'}, 
	\end{equation}
	where $\widehat{f}_{\Gamma'}$ is retrieved from the right-hand-side of (\ref{eq: glob sys}).
	Once this problem is solved, the solution $u_{\Gamma'}$ on the interface $\Gamma'$ is used to recover the solution on the internal degrees of freedom
    \begin{equation*}
	u_{i, \, I} = (K'_{i, \,II})^{-1} \left( f_{i, \, I}  - K'_{i, \, I \Gamma'} u_{ \Gamma'} \right).
    \end{equation*}
	We define the local discrete harmonic part in the sense of $d_i$ as defined in~\eqref{eq: widehat a_i} as follows.
	\begin{definition}
		The discrete harmonic extension $\mathcal H'_i$ in the sense of $d_i$ is defined as:
		\begin{equation*}
			\mathcal H'_i : W_i(\Gamma') \longrightarrow W_i(\Omega'_i), \qquad 
			\begin{cases}
				d_i \left( \mathcal H'_i u_{i, \, \Gamma'}, \vi \right) = 0 		&\forall \vi\in W_i(\Omega'_i) 		\\
				\mathcal H'_i u_{i, \, \Gamma'} = \uii 									&\text{on } \partial \Omega_i	\\
				\mathcal H'_i u_{i, \, \Gamma'} = \uij 								 	&\text{on } E_{ji} \subset \partial \Omega_j
			\end{cases}
		\end{equation*}
	\end{definition}
As in the case of standard elliptic problems (see Ref. \cite{toselli2006domain}), the Schur bilinear form can be defined through the action of the Schur matrix and the bilinear form 
	\begin{equation*}
		d_i (\mathcal H'_i u_{i, \, \Gamma'} , \mathcal H'_i v_{i, \, \Gamma'} ) :=  v_{i, \, \Gamma'}^T S'_i u_{i, \, \Gamma'} = s'_i (u_{i, \, \Gamma'},v_{i, \, \Gamma'}).
	\end{equation*}
	The bilinear form $s'_i(\cdot, \cdot)$ inherits the symmetry and coercivity from the definition of $S'_i$. Furthermore, thanks to Lemma \ref{lemma: ellipticity a_i}, it is possible to bound the energies related to the local Schur complements, 
    \begin{equation}\label{schurminform}
        s'_i (u_{i, \, \Gamma'}, u_{i, \, \Gamma'}) = \min_{v^{(i)}_{| \partial \Omega_i \cap \Gamma'} =  u_{i, \, \Gamma'} } d_i (v^{(i)},v^{(i)})
    \end{equation}
	which allows us to work with discrete harmonic extensions instead of functions defined only on $\Gamma'$.  \\
	
	\paragraph{\bf Restriction operators and scaling}
	Lastly, we need to introduce the restriction operators 
	\begin{eqnarray}
		R_{i, \, \Delta}: W_\Delta(\Gamma') \rightarrow W_{i, \,\Delta} (\Gamma'_i),  		&\qquad R_{\Gamma' \Delta}: W (\Gamma') \rightarrow W_\Delta (\Gamma'), 	\nonumber \\
		R_{i, \, \Pi}: \widehat W_\Pi (\Gamma') \rightarrow W_{i, \, \Pi} (\Gamma'_i),  	 	&\qquad R_{\Gamma' \Pi}: W (\Gamma') \rightarrow \widehat W_\Pi (\Gamma'), 	\nonumber
	\end{eqnarray}
	and the direct sums $R_\Delta = \oplus R_{i,\, \Delta}$, $R_\Pi = \oplus R_{i,\, \Pi} $ and $\widetilde{R}_{\Gamma'} = R_{\Gamma' \Pi} \oplus R_{\Gamma' \Delta}$, which maps $W (\Gamma') $ into $\widetilde W (\Gamma')$. 
	We consider a proper scaling of the dual variables. \\
	{\bf $\rho$-scaling.} The $\rho$-scaling is defined for the cell-by-cell model at each node $x \in \Gamma_i$ through the pseudoinverses
	\begin{equation} \label{eq: pseudoinv}
		\pseudoinv_i (x) := \dfrac{\sigma_i}{\sumjNx \sigma_j} 
	\end{equation}
	where $\mathcal N_x$ denotes the set of indexes of subdomains with $x$ in the closure of the subdomain.
	We observe that $\mathcal N_x$  induces the definition of an equivalence relation that allows the classification of interface degrees of freedom into faces, edges and vertices equivalence classes.  
Recall that the following inequality holds true (see \cite{toselli2006domain}):
\begin{equation}\label{dis_pseu}
\sigma_i (\pseudoinv_i(x))^2 \leq \min{\{\sigma_i,\sigma_j\}} \quad \forall j\in\mathcal N_x.
\end{equation} 
	We define the scaling matrix for each subdomain $\Omega_i$ as 
    \begin{equation}\label{eq: scaling matrices}
    	D_i = \text{diag } 
    		\{ \pseudoinv_i \}_{i=0}^N \,,
    \end{equation}
    i.e. the $i$-th diagonal scaling matrix $D_i$ contains the pseudo-inverses (\ref{eq: pseudoinv}) along the diagonal. 
	
	We finally define the scaled local restriction operators
	\begin{equation*}
		R_{i, \, D, \Gamma'} := D_i R_{i, \, \Gamma'}, 	\qquad 	\qquad R_{i, \, D, \Delta} := R_{i, \, \Gamma' \Delta} R_{i, \, D, \Gamma'} \, ,
	\end{equation*}
	$R_{D, \Delta}$ as direct sum of $R_{i, \, D, \Delta}$ and the global scaled operator 	$\widetilde{R}_{D, \Gamma'} = R_{\Gamma' \Pi} \oplus R_{D, \Delta} R_{\Gamma' \Delta}$.

	\subsection{BDDC preconditioner}
	BDDC is a two-level preconditioner for the Schur complement system $\widehat{S}_{\Gamma'} u_{\Gamma'} = \widehat{f}_{\Gamma'}$.
	If we partition the degrees of freedom of the interface $\Gamma'$ into those internal ($I$) and those dual ($\Delta$), the matrix $\mathcal K'_i$ can be written as
	\begin{equation*}
		\mathcal K'_i = 
		\begin{bmatrix}
			K'_{i, \, II} 					& K'_{i, \, I\Gamma'} \\
			K'_{i, \, \Gamma' I}  		    & K'_{i, \, \Gamma' \Gamma'}
		\end{bmatrix}
		=
		\begin{bmatrix}
			K'_{i, \, II}			& K'_{i, \, I \Delta} 			& K'_{i, \, I \Pi} \\
			K'_{i, \, \Delta I} 	& K'_{i, \, \Delta \Delta} 		& K'_{i, \, \Delta \Pi}	\\
			K'_{i, \, \Pi I} 		& K'_{i, \, \Pi \Delta}			& K'_{i, \, \Pi \Pi}
		\end{bmatrix}.
	\end{equation*} 
	It is possible to define the BDDC preconditioner using the restriction operators as 
	\begin{equation*}
		M^{-1}_\text{BDDC} = \widetilde{R}_{D, \Gamma'}^T (\widetilde S_{\Gamma'})^{-1}  \widetilde{R}_{D, \Gamma'}, 	\qquad 		\widetilde S_{\Gamma'} = \widetilde{R}_{\Gamma'} S' \widetilde{R}_{\Gamma'}^T,
	\end{equation*}
	where the action of the inverse of $\widetilde{S}_{\Gamma'}$ can be evaluated with a block-Cholesky elimination procedure
	\begin{equation*}
		\widetilde S_{\Gamma'}^{-1} = \widetilde{R}_{\Gamma \Delta }^T \left( \sum_{i=0}^{N} 
		\begin{bmatrix}
			0 		&R_{i, \, \Delta}^{T}
		\end{bmatrix}
		\begin{bmatrix}
			K'_{i, \, II} 					& K'_{i, \, I \Delta} 			  \\
			K'_{i, \, \Delta I}  			& K'_{i, \, \Delta \Delta}
		\end{bmatrix}^{-1}
		\begin{bmatrix}
			0 	\\		R_{i, \, \Delta}
		\end{bmatrix}
		\right) \widetilde{R}_{\Gamma' \Delta } + \varPhi S_{\Pi \Pi}^{-1} \varPhi .
	\end{equation*}
	In this way, the first term is the sum of local solvers on each substructure $\Omega'_i$, while the latter is a coarse solver for the primal variables, where
	\begin{equation*}
		\begin{split}
		\varPhi & =  R_{\Gamma' \Pi}^T - R_{\Gamma' \Delta}^T \sum_{i=0}^{N}
		\begin{bmatrix}
			0 		&R_{i, \Delta}^{T}
		\end{bmatrix}
		\begin{bmatrix}
			K'_{i, \, II} 					& K'_{i, \, I \Delta} 			  \\
			K'_{i, \, \Delta I}  			& K'_{i, \, \Delta \Delta}
		\end{bmatrix}^{-1}
		\begin{bmatrix}
			K'_{i, \, I \Pi} 	\\		R_{i, \, \Delta \Pi}
		\end{bmatrix}
		R_{i, \, \Pi} ,\\
		S_{\Pi \Pi } & =  \sum_{i=0}^{N}  R_{i, \, \Pi}^T 
		\left( K'_{i, \, \Pi \Pi} - 
		\begin{bmatrix}
			K'_{i, \, \Pi I}		    & K'_{i, \, \Pi \Delta}   
		\end{bmatrix}
		\begin{bmatrix}
			K'_{i, \, II} 				& K'_{i, \, I \Delta} 			  \\
			K'_{i, \, \Delta I} 		& K'_{i, \, \Delta \Delta}
		\end{bmatrix}^{-1}
		\begin{bmatrix}
			K'_{i, \, I \Pi} 	\\		R_{i, \, \Delta \Pi}
		\end{bmatrix}
		\right) R_{i, \, \Pi}.
	\end{split}
	\end{equation*}
	
	\subsection{Technical results}
	We observe first that the bilinear form $d_h(\cdot, \cdot)$ induces a broken norm on $V(\Omega)$:
	\begin{equation*}
		\| u \|^2_h = d_h(u,u) = \sumN \left\{ \tau \sigma_i \normL{\grad u_i}{\Omega_i}^2 + \dfrac{1}{2} \sumEij \int_{\Eij} C_m (u_i - u_j)^2 \right\} .
	\end{equation*}
	Furthermore, it is possible to prove the following Lemma.
	\begin{lemma}\label{lemma: ellipticity a_i}
		Consider the bilinear form $d_i(u,v) = \tau \, a_i(u,v) + p_i(u,v),$ where $a_i(\cdot, \cdot)$ and $p_i(\cdot, \cdot)$ have been defined in (\ref{eq: a_i p_i}). The following bounds hold
		\begin{align*}
			d_i (u_i, u_i) &\leq \tau \sigma_M \seminormH{u_i}{i}^2 + \sumEij \frac{C_m}{2} \normL{u_i - u_j}{\Eij}^2,	\\
			d_i (u_i, u_i) &\geq \tau \sigma_m \seminormH{u_i}{i}^2 + \sumEij \frac{C_m}{2} \normL{u_i - u_j}{\Eij}^2 ,
		\end{align*}
		for all $u_i \in V_i(\Omega_i)$, with $\sigma_m$ and $\sigma_M$ the minimum and maximum values of $\sigma_i$ over all the subdomains.
	\end{lemma}
	
	This result can be easily obtained by considering the continuity and coercivity of the standard Laplacian bilinear form $a_i$.
	
	\begin{remark}
		We will write $A \lesssim B$ whenever $A \leq c B$ with $c$ a constant independent of the various parameters of the problem (e.g., the subdomain diameter $H$, the mesh size $h$, the time step $\tau$ and the conductivity coefficients); similarly, we will write $A \sim B$ whenever it holds $A \lesssim B$ and $B \lesssim A$.
	\end{remark}
	
	We assume that the cardiac domain $\Omega_i \in \mathbb R^2$ is a bounded domain with Lipschitz continuous boundary; moreover, suppose that $\Gamma_i \subset \partial \Omega_i$ has non-vanishing two-dimensional measure and is relatively open with respect to $\partial \Omega_i$. 
	We will work with Sobolev spaces, considering mainly the space $H^{1/2}(\Gamma_i)$ of functions $u \in L^2(\Gamma_i)$ with finite semi-norm $| u |_{H^{1/2}(\Gamma_i)} < \infty$, where
	\begin{equation*}
		\seminormHhalf{u}{\Gamma_i}^2 = \iint_{\Gamma_i} \dfrac{|u(x) - u(y)|^2}{|x-y|^2} ds_x \, ds_y,
	\end{equation*}
	with $ds_\ast$ the differential of the arc-length along $\ast$-direction, and the associated norm
	\begin{equation*}
		\| u \|_{H^{1/2} (\Gamma_i)}^2 = \normL{u}{\Gamma_i}^2 + \seminormHhalf{u}{\Gamma_i}^2.
	\end{equation*}
	We define the space
	\begin{equation*}
		H^{1/2}_{00} (\Gamma_i) = \left\{ u \in H^{1/2}(\Gamma_i): \, \mathcal E_\text{ext} u \in H^{1/2} (\partial \Omega_i) \right\},
	\end{equation*}
	as the set of functions for which the zero extension from $\Gamma_i$ to $\partial \Omega_i$ by the extension operator $\mathcal E_\text{ext}$ 
	\begin{equation*}
		\mathcal E_\text{ext}: \Gamma_i \longrightarrow \partial \Omega_i 
		\qquad
		\mathcal E_\text{ext} u = 
		\begin{cases}
			0			&\text{on } \partial \Omega_i \backslash \Gamma_i,	\\
			u			&\text{on } \Gamma_i.
		\end{cases}
	\end{equation*}
	
	We report here a theoretical result that will be employed in the convergence analysis. This result can be found in \cite[Lemma 3.3]{widlund1988iterative}.

    \begin{lemma} \label{lemma: edge estimate}
        Let $u$ be a finite element, discrete harmonic function on a substructure $\Omega_i$. Consider the function $u - I_H u$, where $I_H u$ is the piecewise linear interpolant of $u$ on a coarse triangulation. Then, it holds
        \begin{equation*}
            \sum_{j \in \mathcal E^0_i} \| u - I_H u \|^2_{H^{1/2}_{00} (\Eij)} \leq C \left( 1 + \log \dfrac{H}{h} \right)^2 |u|^2_{H^1(\Omega_i)}
        \end{equation*}
        or, equivalently,
        \begin{equation*}
            \sum_{j \in \mathcal E^0_i} \| u - I_H u \|^2_{H^{1/2}_{00} (\Eij)} \leq C \left( 1 + \log \dfrac{H}{h} \right)^2 \seminormHhalf{u_{|_{\Gamma_i}}}{\Gamma_i}^2.
        \end{equation*}
    \end{lemma}
	
	
	Additionally, we recall some classical and useful theoretical results that will be considered throughout the proof. They can be found in Appendix A of \cite{toselli2006domain}.
	
	\begin{lemma}\label{lemma: 3.1}
		Let $\Gamma_i \subset \partial \Omega_i$. There exists two constants, such that for $u \in H^{1/2}_{00} (\Gamma_i)$
		\begin{equation*}
			c \| \mathcal E_\text{ext} u \|^2_{H^{1/2} (\partial \Omega_i)} \leq \normHzero{u}{\Gamma_i}^2 \leq C \| \mathcal E_\text{ext} u \|^2_{H^{1/2} (\partial \Omega_i)} .
		\end{equation*}
	\end{lemma}
	
	\begin{theorem}[Trace theorem]
		Let $\Omega_i$ be a polyhedral domain. Then,
		\begin{equation*}
			\seminormHhalf{u}{\Gamma_i} \sim \seminormH{\mathcal H^\Delta_i u_\Gamma}{i}^2.
		\end{equation*}
	\end{theorem}

	\section{Bound for projection operator}\label{sec: theoretical convergence}
	The main tool in the theory of dual-primal iterative substructuring methods is given by the projection operator $P_D: W (\Gamma') \rightarrow W (\Gamma')$, whose action on a given $u \in W (\Gamma')$ can be locally defined for this particular application as
	\begin{equation} \label{eq: projection P_D}
	\left( P_D u \right)_i (x) = 
	\left( \sumjNx \pseudoinv_j \left( \uii(x) - \uji(x) \right), \, \left\{ \pseudoinv_j \left( \uij(x) - \ujj(x) \right) \right\}_{j \in \mathcal E_i^0} \right),	
	\end{equation}
	for all $x \in \Gamma_i$, 
	where $\uii$ and $\ujj$ are the values of $u_i \in \Omega_i$ and $u_j \in \Omega_j$ on the boundaries $\Eij$ and $\Eji$, respectively, while $\uij$ and $\uji$ are the values of $u_i$ and $u_j$ on $\Eji$ and $\Eij$, respectively (see, e.g., Figure~\ref{fig: subdomain dofs}). 
    In other words, the operator $P_D$ can be viewed as the action of the scaling operator introduced through \eqref{eq: pseudoinv}. For the projection $E_D=I-P_D$, it holds
    \begin{align*}
        (E_D u)_{i,i} (x) &= (E_D u )_{j,i} (x) 
        \qquad
        \forall j \in \mathcal N_x \text{ and } \forall x \in \Gamma_i, \\
        (E_D u)_{i,j} (x) &= (E_D u )_{j,j} (x) 
        \qquad
        \forall j \in \mathcal E_i^0 \text{ and } \forall x \in \Gamma_i,
    \end{align*}
    i.e., $E_D$ maps onto function continuous across the BDDC splitting, but still possibly discontinuous across cell membranes.
    We refer to \cite{toselli2006domain} for further details.
	
	We focus on the estimate of the norm of this operator with respect to the seminorm $| \cdot |_{S'}^2$ defined as 
	\begin{equation*}
		| v |_{S'}^2 := \sumN \seminormS{v_i}{i}^2,
		\qquad
		\seminormS{v_i}{i}^2 := v_i^T S'_i v_i,
		\qquad
		\text{for } v_i \in W (\Gamma'_i).
	\end{equation*}	
	
	The core of the estimate relies on the following Lemma. 
	
	\begin{lemma}\label{lemma: projection lemma}
		Let the primal space $\widetilde W_\Pi (\Gamma')$ be spanned by the vertex nodal finite element functions and $h = \mathcal O(\tau)$. If the projection operator $P_D$ is scaled by the $\rho$-scaling, then
		 \begin{equation*}
		 	|P_D u |_{S'}^2 \leq C \, \varPsi (\tau, h) \left( 1 + \log \dfrac{H}{h} \right)^2  |u|_{S'}^2
		 \end{equation*}
		 holds for all $u \in \widetilde{W} (\Gamma')$
		 with $C$ constant independent of all the parameters of the problem and with 
\[
\varPsi(\tau, h) = 1+\frac{2\,\sigma_M\,\tau}{C_m\,h}.
\]
		 Here, $\sigma_M$ is the maximum value of $\sigma_i$ over all the subdomains, $\tau$ the time step size, and $h$ the mesh size.
	\end{lemma}
	
	\begin{proof}
	As usually done in the substructuring framework, it is sufficient to work only on the local bounds, since it holds that	
	$$|P_D u|_{S'}^2 = \sumN{ \seminormS{R_{\partial \Omega_i} P_D u }{i}^2 }. $$
	Consider a function $u \in \widetilde{W} (\Gamma')$; then the action of the projection operator can be written as (\ref{eq: projection P_D}),
	\begin{align*}
		\vi &= \left( P_D u \right)_i 
		= \left( \sumjNx I^h \pseudoinv_j \left( \uii - \uji \right), \, \left\{ I^h \pseudoinv_j \left( \uij - \ujj \right) \right\}_{j \in \mathcal E_i^0}\right)	\\
		&= \sumE I^h \left( \ThetaE \, v_i \right) + \sumV I^h \left( \ThetaV \, v_i \right)
	\end{align*}
	where $\Theta_\ast = (\theta_\ast^{i}, \theta_\ast^{j})$ is the characteristic finite element function associated to the edges or the vertices $\ast = \{ \mathcal E, \mathcal V \}$ and $I^h$ is the usual finite element interpolant
	\begin{equation*}
		I^h: \mathcal C^0 (\Omega'_i) \longrightarrow W_i(\Omega'_i),
		\qquad
		\forall i=0,\dots,N.
	\end{equation*}
	Since we have included the vertices in the primal space, the vertex term vanishes. Therefore, we have only to estimate the contributions given by the edges
	\begin{equation*}
		\seminormS{ \vi }{i}^2 \lesssim \sumE \seminormS{ I^h \left( \ThetaE \, \vi \right)}{i}^2.
	\end{equation*}
	Thus,
	\begin{align*}
		\seminormS{ I^h (\ThetaE \vi )}{i}^2 &= s'_i \left( I^h (\ThetaE \vi), \, I^h (\ThetaE \vi) \right)
			= d_i \left( \mathcal H'_i I^h (\ThetaE \vi), \, \mathcal H'_i I^h (\ThetaE \vi) \right)		\\
			&= \tau \sigma_i \seminormH{\mathcal H_i^\Delta I^h (\ThetaE \vi)}{i}^2 + \sumEij \frac{C_m}{2} \normL{I^h (\vi - \vj)}{\Eij}^2,\numberthis \label{eq: eqn_2_terms}
	\end{align*}
	where we used the ellipticity property of Lemma \ref{lemma: ellipticity a_i}.	We proceed by estimating each of the above terms separately.
    In order to apply Lemma \ref{lemma: edge estimate}, we define $w_i$ as the linear function on $\Eij$ that assumes the same values of $\uii$ and $\uji$ at the endpoints of $\Eij$ and $w_j$ as the linear function on $\Eji$ that assumes the same values of $\uij$ and $\ujj$ at the endpoints of $\Eji$.   
    Then, by applying a Trace theorem, Lemma \ref{lemma: 3.1}, (\ref{dis_pseu}) and Lemma \ref{lemma: edge estimate}
    \begin{align*}
      \sigma_i & \seminormH{\mathcal H_i^\Delta I^h (\ThetaE \vi)}{i}^2 \lesssim \sigma_i \seminormHhalf{I^h (\ThetaE \vi)}{\Gamma_i}^2 \\ 				
		      &\lesssim \sum_{j \in \mathcal E^0_i} \sigma_i \normHzero{I^h \ThetaE (\pseudoinv_j (\uii - \uji) )}{\Eij}^2 \\
              & \leq \sum_{j \in \mathcal E^0_i} \sigma_i \left( \pseudoinv_j  \right)^2 \normHzero{I^h \ThetaE (\uii - w_i + w_i - w_j + w_j - \ujj + \ujj - \uji)}{\Eij}^2
\\		
	    &\lesssim \sum_{j \in \mathcal E^0_i} \sigma_i \normHzero{I^h \ThetaE (\uii - w_i)}{\Eij}^2 + \sum_{j \in \mathcal E^0_i} \sigma_j \normHzero{I^h \ThetaE (\ujj - w_j)}{\Eij}^2 
\\
	    &+ \sum_{j \in \mathcal E^0_i} \sigma_i \left( \pseudoinv_j  \right)^2 \normHzero{I^h \ThetaE (\uji-w_i-(\ujj-w_j))}{\Eij}^2
\\  
            &\lesssim \left( 1 + \log \dfrac{H}{h} \right)^2 \left[\sigma_i | \mathcal H^\Delta_i \uii |_{H^1(\Omega_i')}^2 + \sigma_j | \mathcal H^\Delta_j \ujj |_{H^1(\Omega_j')}^2 \right]       
\\ 
            & + \sum_{j \in \mathcal E^0_i} \sigma_i \left( \pseudoinv_j  \right)^2 \normHzero{I^h \ThetaE (\uji-w_i-(\ujj-w_j))}{\Eij}^2
\\
	    &\lesssim \left( 1 + \log \dfrac{H}{h} \right)^2 \left[ \dfrac{1}{\tau} \seminormS{u_i}{i}^2 + \dfrac{1}{\tau} \seminormS{u_j}{j}^2 \right] 
\\
            & + \sum_{j \in \mathcal E^0_i} \sigma_i \left( \pseudoinv_j  \right)^2 \normHzero{I^h \ThetaE (\uji-w_i-(\ujj-w_j))}{\Eij}^2. \numberthis \label{eq: bound_I}
    \end{align*}
The last term in (\ref{eq: bound_I}) is estimated using the same techniques adopted in \cite[Proof of Lemma 4.5]{dryja2013analysis}, thus
we have
\[
\normHzero{I^h \ThetaE (\uji-w_i-(\ujj-w_j))}{\Eij}^2 \lesssim \left(1+\log \dfrac{H}{h} \right)^2 | \mathcal H^\Delta_j \ujj |_{H^1(\Omega_j')}^2 + \frac{1}{h} \normL{\uji - \ujj}{\Eij}^2, 
\]
from which it finally follows
\begin{equation}\label{eq: bound_I_bis}
\sigma_i  \seminormH{\mathcal H_i^\Delta I^h (\ThetaE \vi)}{i}^2  \lesssim \left( 1 + \log \dfrac{H}{h} \right)^2 \left[ \dfrac{1}{\tau} \seminormS{u_i}{i}^2 + \left(\dfrac{1}{\tau}+\frac{2\,\sigma_M}{C_m\,h}\right) \seminormS{u_j}{j}^2 \right]. 
\end{equation}
	Regarding the second term in (\ref{eq: eqn_2_terms}), since
	$\vii = \uii - \uji$ and $\vij = \uij - \ujj$, we have
	\begin{align*}
	&	\normL{I^h (\vi - \vj)}{\Eij}^2 = \int_{\Eij} \left( \vii - \vij \right)^2 		
		    = 			\int_{\Eij} \left( \uii - \uij - \left( \uji - \ujj \right) \right)^2				\\
			&= 			\normL{\uii - \uij - \left( \uji - \ujj \right)}{\Eij}^2						
			\lesssim 	\normL{\uii - \uij}{\Eij}^2 + \normL{\uji - \ujj}{\Eij}^2 						\\
			&\lesssim	d_i \left( \mathcal H'_i u_i, \, \mathcal H'_i u_i \right) + d_j \left( \mathcal H'_j u_j, \, \mathcal H'_j u_j \right)	
			\lesssim 	\seminormS{u_i}{i}^2 + \seminormS{u_j}{j}^2.
			\numberthis \label{eq: bound_II}
	\end{align*}
	In conclusion, adding (\ref{eq: bound_I}) and (\ref{eq: bound_II}), we obtain
	\begin{align*}
		\seminormS{\vi}{i}^2 &= \seminormS{(P_D u)_i}{i}^2 \leq \sumE \seminormS{I^h \ThetaE \vi}{i}^2 	\\
			&\lesssim 		\tau \left( 1 + \log \dfrac{H}{h} \right)^2 \left[ \dfrac{1}{\tau} \seminormS{u_i}{i}^2 + \left(\dfrac{1}{\tau}+\frac{2\,\sigma_M}{C_m\,h}\right) \seminormS{u_j}{j}^2 \right]  
			+ \sumEij \dfrac{C_m}{2} \left( \seminormS{u_i}{i}^2 + \seminormS{u_j}{j}^2 \right)    	\\
			&\lesssim 	\left( 1 + \log \dfrac{H}{h} \right)^2 \left[1+\frac{2\,\sigma_M\,\tau}{C_m\,h}\right] \left( \seminormS{u_i}{i}^2 + \seminormS{u_j}{j}^2 \right),
	\end{align*}
	and, by adding all the local contributions, we have the global bound
	\begin{equation*}
		|P_D u |_{S'}^2 \lesssim \left( 1 + \log \dfrac{H}{h} \right)^2 \left[1+\frac{2\,\sigma_M\,\tau}{C_m\,h}\right] |u|_{S'}^2.
	\end{equation*}

	\end{proof}

	 \begin{remark}
	 	The hypothesis of $h = \mathcal O(\tau)$ is needed since the quantity
	 	\begin{equation*}
	 		\varPsi (\tau, h) = 1+\frac{2\,\sigma_M\,\tau}{C_m\,h} 
	 	\end{equation*}
	 	can otherwise grow uncontrolled when the mesh size decreases, i.e. $h \rightarrow 0$. This is only a theoretical limitation, since extensive numerical tests show that this dependency does not hold.
	 \end{remark}
	
	\section{Numerical tests} \label{sec: numerical results}
	We report here extensive numerical tests showing scalability, quasi-optimality and robustness with respect to the time step $\tau$ of the proposed preconditioned solver.
	We consider a Matlab implementation, which takes into account a two-dimensional rectangular geometry where the extracellular domain frames the considered cells. This code considers a linear gap junction between cells and the Aliev-Panfilov ionic model \cite{aliev1996simple} for the update of the gating variables. The values of $\sigma_i$ have been chosen such that the extracellular space $\Omega_0$ presents different a conductivity with respect to all the cells $\Omega_i$, $i=1,\dots,N$ :
    \begin{equation*}
        \sigma_0 = 2e-3,
        \qquad
        \sigma_i = 3e-3,
        \quad \forall i = 1, \dots, N.
    \end{equation*}
    The time step is fixed to $\tau = 0.05$ ms, except for Section \ref{sec: dependence tau} where this parameter is varied, and  $100$ steps are taken, corresponding to a total time interval of $[0,5]$ ms. The external current needed for the activation is applied at one corner of the domain for $1$ ms, with an intensity of 50 mA/cm$^2$. The initial value for the potential $u_j(0)$ is set to $-85$ mV. We also show the robustness of the proposed solver over an activation-recovery time interval of $[0,300]$ ms in Section \ref{sec: time evolution}.

    We solve the Schur complement linear system by the Conjugate Gradient method, either unpreconditioned (denoted in the numerical test by \emph{Schur}) or preconditioned by the BDDC (denoted by \emph{BDDC}). We also compared these solvers with the unpreconditioned solution of the full system using the Conjugate Gradient (denoted by \emph{CG}). 
    The stopping criterion compares the $L^2$-norm of the relative (preconditioned) residual with a fixed tolerance of $10^{-6}$.
	All the tests have been performed on a workstation equipped with an Intel i9-10980XE CPU with 18 cores running at 3.00 GHz.
	
	\subsection{Time evolution} \label{sec: time evolution}
	In this Section we show the robustness of the proposed solver during a time interval of $[0,300]$ ms, discretized with fixed time step of $\tau=0.05$ ms (thus considering 6'000 steps). 
	We consider $14\times14$ cells, each represented as an elongated rectangle and discretized with $24 \times 4$ finite elements.
	%
	We show in Figure \ref{fig: time evolution} the first 6 time iterations, where it is possible to appreciate the activation of the first group of cells and the subsequent propagation. 
 The discrete and step colormap could be interpreted as an indicator of a poor space discretization, but this phenomenon is due to the discrete nature of the model, where the cells are represented as independent entities.
	As a matter of fact, if we plot the transmembrane voltage at two fixed points, as shown in Figure \ref{fig: time evolution specific points}, we observe that the usual shape of the action potential appears.
	\begin{figure}
	    \centering
	    \includegraphics[width=.7\textwidth]{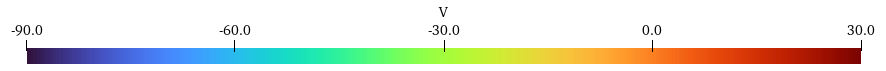}   \\
	    \vspace{3mm}
	    t = 0.0 ms  \\
	    \vspace{1mm}
	    \includegraphics[width=.8\textwidth]{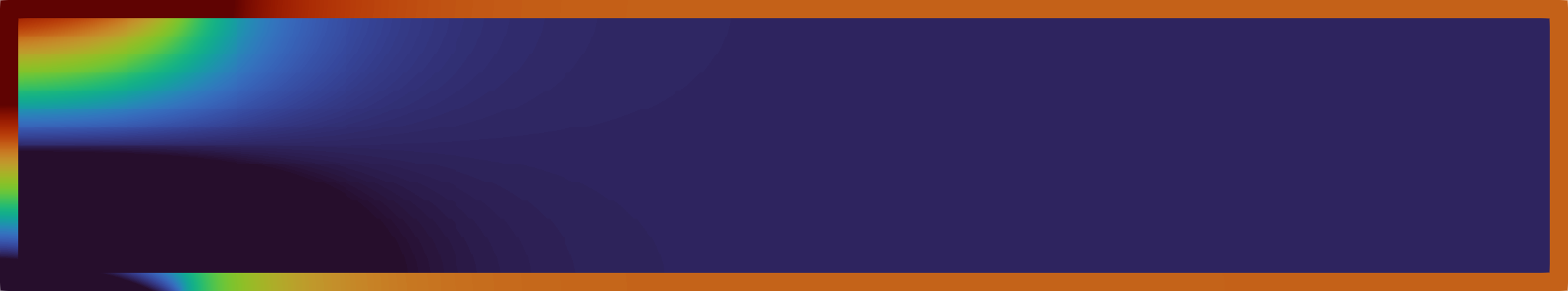}   \\
	    \vspace{1mm}
	    t = 1.0 ms  \\
	    \vspace{1mm}
	    \includegraphics[width=.8\textwidth]{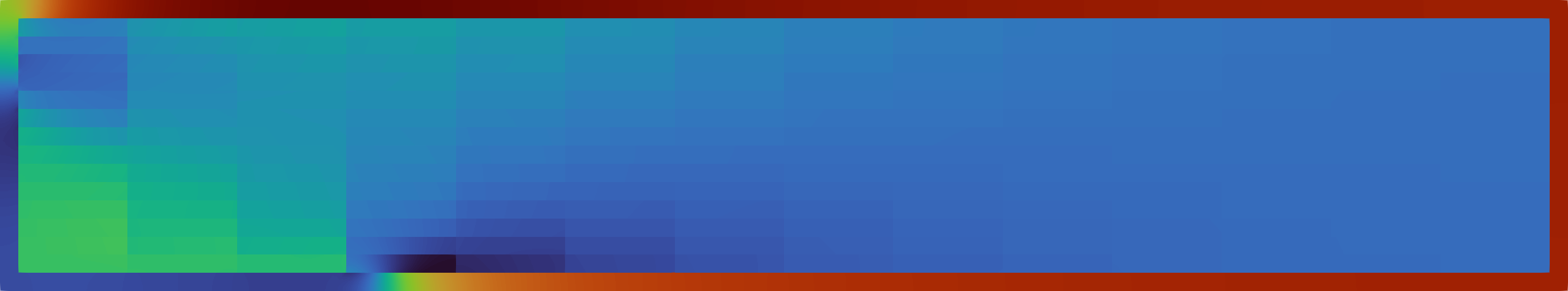}   \\
	    \vspace{1mm}
	    t = 2.0 ms  \\
	    \vspace{1mm}
	    \includegraphics[width=.8\textwidth]{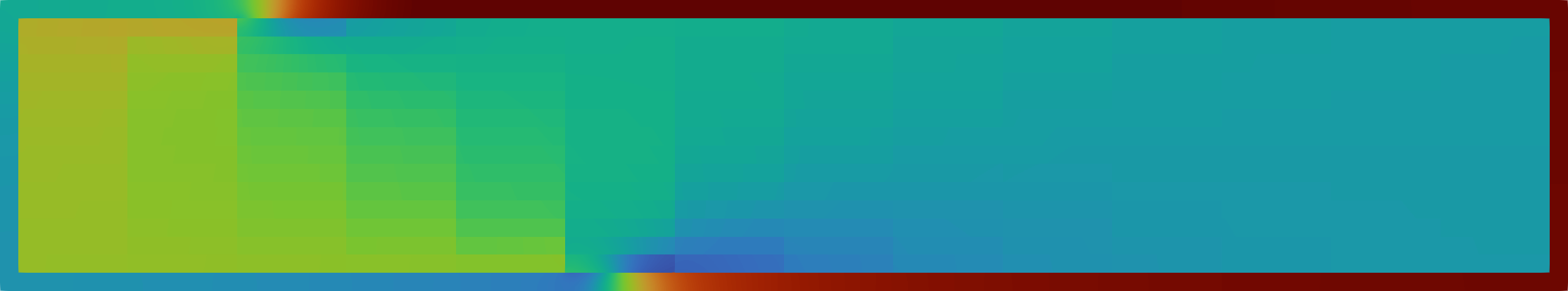}   \\
	    \vspace{1mm}
	    t = 3.0 ms  \\
	    \vspace{1mm}
	    \includegraphics[width=.8\textwidth]{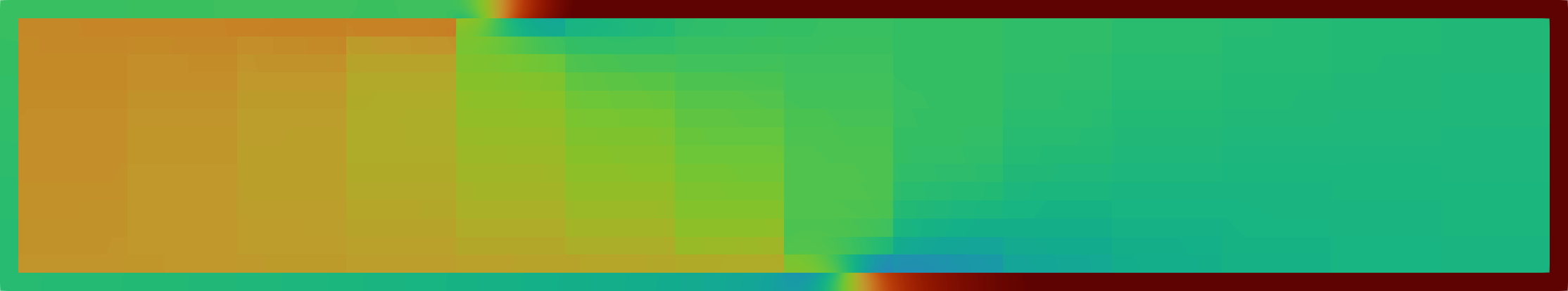}   \\
	    \vspace{1mm}
	    t = 4.0 ms \\
	    \vspace{1mm}
	    \includegraphics[width=.8\textwidth]{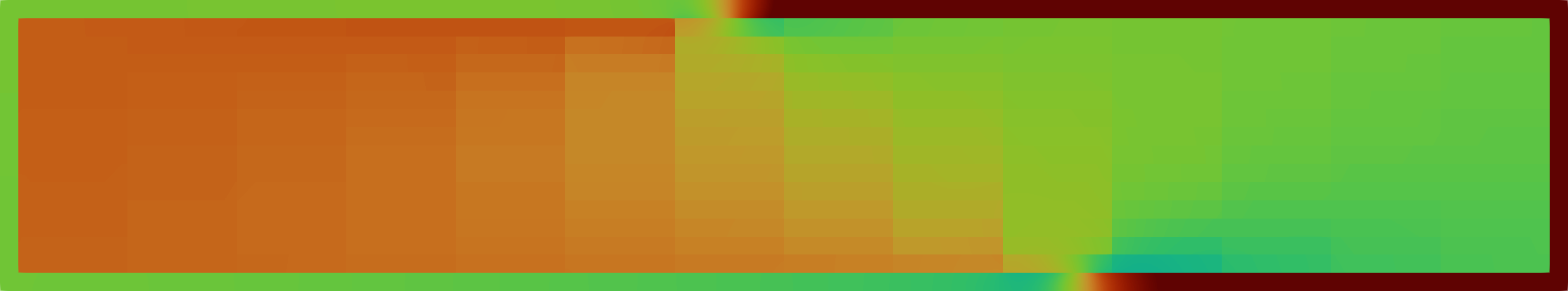}   \\
	    \vspace{1mm}
	    t = 5.0 ms \\
	    \vspace{1mm}
	    \includegraphics[width=.8\textwidth]{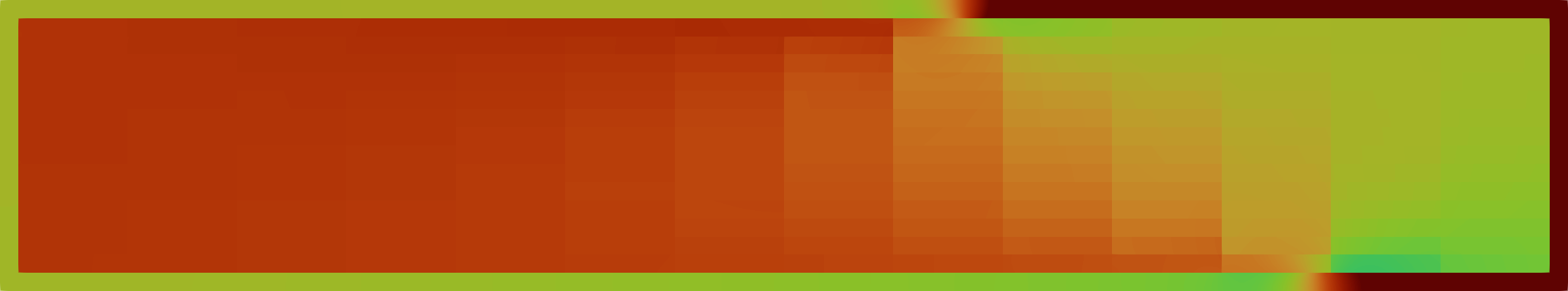}   \\
	    \caption{Time evolution at different time steps.}
	    \label{fig: time evolution}
	\end{figure}
	
	\begin{figure} 
     	\centering
     	\begin{overpic}[abs,unit=1mm,width=.85\linewidth, trim=0 -1cm 0 0]{plots/timeframe/plot_196cells_000.png}
            \put(10,2){\color{white} \bf \large $\bullet$ A}
            \put(120,20){\color{white} \bf \large $\bullet$  B}
        \end{overpic}
     	\begin{tikzpicture}
     	    \begin{axis}
     	        [
                width=.8\textwidth,
                height=2.5cm,
                scale only axis,
                scaled ticks = false,
                tick label style={/pgf/number format/fixed, /pgf/number format/precision=5},
                tick label style={font=\tiny},
                xlabel={Time (ms)},
     			xmin=-5, xmax=305,
     			ymin=-110, ymax=50,
                unbounded coords=jump]
                
                \addplot+[cyan, line width=1.5pt, mark=none] table [x=time, y=tap_1, col sep=comma, filter discard warning=false, unbounded coords=discard, each nth point=20] {results/time_plot.csv};
                
                \addplot+[green!50!black, dashed, line width=1.5pt, mark=none] table [x=time, y=tap_2, col sep=comma, filter discard warning=false, unbounded coords=discard, each nth point=20] {results/time_plot.csv};
                
                \legend{A, B}
   
     	    \end{axis}
		\end{tikzpicture}
        \caption{Time evolution on $[0,300]$ ms of transmembrane voltages (mV) at two fixed points. }
        \label{fig: time evolution specific points}
     \end{figure}
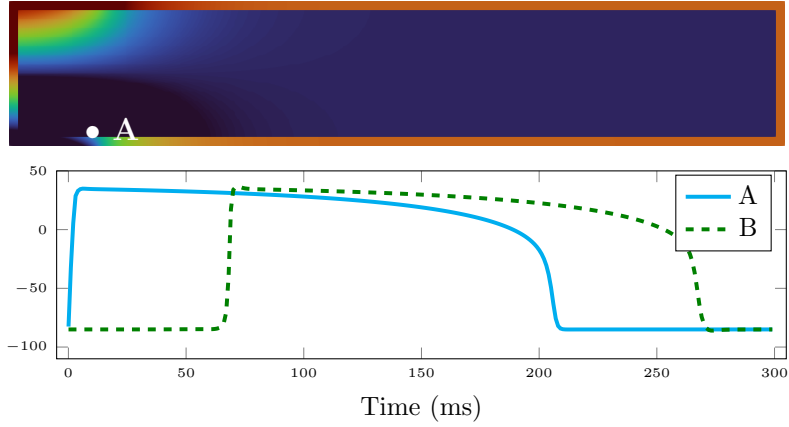
	
	Turning to the numerics, we report the trend of the number of linear iterations, the condition number and the CPU timings (in seconds) over the considered time interval. We show the performance of the unpreconditioned conjugate gradient (\emph{Schur}) as well as the BDDC preconditioned solver (\emph{BDDC}) for the Schur complement system and the case where the full system is solved (\emph{CG}), see Figure \ref{fig: time evolution iterations and condnum}.
	The BDDC preconditioner is clearly superior, both in terms of number of linear iteration and condition number, since for the Schur and CG cases we have higher values.
    Moreover, the BDDC preconditioner reacts more sensitive to changes to the right hand side.
	
	\begin{figure}
     	\centering
     	\begin{tikzpicture}
     	    \begin{groupplot}
     	        [
     	        group style={group size=1 by 3, vertical sep=0pt, ylabels at=edge left,
     	        	x descriptions at=edge bottom},
                width=.8\textwidth,
                height=2.5cm,
                scale only axis,
                scaled ticks = false,
                tick label style={/pgf/number format/fixed, /pgf/number format/precision=5},
                tick label style={font=\tiny},
                xlabel={Time (ms)},
     			xmin=-5, xmax=305,
                unbounded coords=jump]
                
                \nextgroupplot[legend style={at={(0.99,0.31)},anchor=east}, ymax=34]
                \addplot+[cyan, line width=1.5pt, mark=none] table [x=time, y=iter_bddc, col sep=comma, filter discard warning=false, unbounded coords=discard, each nth point=10] {results/time_plot.csv};
                
                \addplot+[cyan, dashed, line width=1.5pt, mark=none] table [x=time, y=cond_bddc, col sep=comma, filter discard warning=false, unbounded coords=discard, each nth point=10] {results/time_plot.csv};
                
                \legend{BDDC iter, BDDC $k_2$};
                
                \nextgroupplot[height=2.5cm, legend style={at={(0.99,0.165)},anchor=east}, ymin=180]
   				\addplot+[orange, line width=1.5pt, mark=none] table [x=time, y=iter_schur, col sep=comma, filter discard warning=false, unbounded coords=discard, each nth point=10] {results/time_plot.csv};
   				
   				\addplot+[green!50!black, line width=1.5pt, mark=none] table [x=time, y=iter_pcg, col sep=comma, filter discard warning=false, unbounded coords=discard, each nth point=10] {results/time_plot.csv};
   				
   				\legend{Schur iter, CG iter};
   				
   				\nextgroupplot[height=2.5cm, legend style={at={(0.99,0.185)},anchor=east}, ymax=85000]
   				\addplot+[orange, dashed, line width=1.5pt, mark=none] table [x=time, y=cond_schur, col sep=comma, filter discard warning=false, unbounded coords=discard, each nth point=10] {results/time_plot.csv};
   				
   				\addplot+[green!50!black, dashed, line width=1.5pt, mark=none] table [x=time, y=cond_pcg, col sep=comma, filter discard warning=false, unbounded coords=discard, each nth point=10] {results/time_plot.csv};
   				
   				\legend{Schur $k_2$, CG $k_2$};
   				
     	    \end{groupplot}
		\end{tikzpicture}
        \caption{Time evolution on $[0,300]$ ms of linear iterations (iter) and condition number ($k_2$) for the three solver under consideration (BDDC, Schur, CG). }
        \label{fig: time evolution iterations and condnum}
     \end{figure}
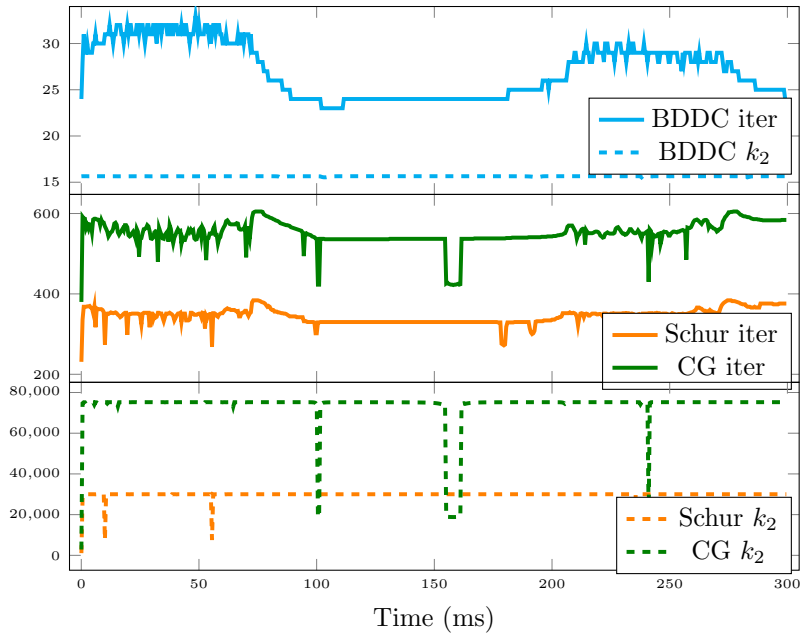

	\subsection{Scalability}
	For the scalability tests, we consider a time interval of $[0,5]$ ms (for a total of 100 fixed time steps). 
	We increase the number of cells from $2\times2$ to $32\times32$, each of which is discretized by $24\times4$ finite elements, resulting in more than $10^5$ degrees of freedom in the largest case.
	The results are collected in Table \ref{tab: weak scal}, where we report the condition number ($k_2$) and the number of linear iterations (it) at the final time step. We consider the unpreconditioned conjugate gradient (\emph{Schur}) as well as the BDDC preconditioned solver (\emph{BDDC}) for the Schur complement system and the case where the full system is solved (\emph{CG}). The BDDC is scaled with a standard $\rho$-scaling procedure.
	The advantage of the BDDC preconditioner with respect to the others is clear: the number of linear iterations mildly increases at the beginning but remains constant afterwards, while for the other cases it increases uncontrolled. The same conclusion holds for the condition number.
	
	\begin{table}[!ht]
	    \centering
	    \caption{\emph{Weak scalability.} Weak scalability tests on $[0,5]$ ms, fixed time step $\tau=0.05$. Increasing number of cells from $2\times2$ to $32\times32$, each discretized with $24\times4$ finite elements. Reported number of linear iterations (it) and condition number ($k_2$) at time $t=5$ ms.}
	    \label{tab: weak scal}
	    \begin{tabular}{ccccccc}
	        \toprule
	        \multirow{2}{*}{nb of cells} &\multicolumn{2}{c}{BDDC}    &\multicolumn{2}{c}{Schur}     &\multicolumn{2}{c}{CG}    \\
	                               &$k_2$              &it                 &$k_2$              &it         &$k_2$              &it       \\      
	        \midrule
	        $2\times2$      &9.28       &15         &57.71          &62         &2.20e+03       &124            \\
	        $4\times4$      &12.35      &21         &2.34e+03       &114        &7.09e+03       &204            \\
	        $8\times8$      &14.77      &27         &9.36e+03       &210        &2.50e+04       &355            \\
	        $12\times12$    &15.40      &28         &2.05e+04       &305        &5.26e+04       &494            \\
	        $16\times16$    &15.82      &29         &3.52e+04       &395        &8.94e+04       &624            \\
	        $20\times20$    &16.02      &30         &5.13e+04       &473        &1.32e+05       &756            \\
	        $24\times24$    &16.15      &30         &6.97e+04       &547        &1.79e+05       &877            \\
	        $28\times28$    &16.24      &30         &8.74e+04       &610        &1.89e+05       &920            \\
	        $32\times32$    &16.32      &30         &1.09e+05       &677        &1.78e+05       &932            \\
	        \bottomrule
	    \end{tabular}
	\end{table}

	\subsection{Optimality tests}
    We report here the results of quasi-optimality tests, where we fix the number of cells to $4\times4$ and we increase the finite element size.
    The simulation time considers $100$ fixed time steps, over the interval $[0,5]$~ms.
    As for the scalability, we observe from results in Table \ref{tab: optimality matlab} that our theoretical result \ref{lemma: 3.1} is confirmed numerically. As a matter of fact, the condition number ($k_2$) for the BDDC mildly increases, following the expected polylogarithmic trend, while for the other cases this parameter increases uncontrolled. 
    The number of linear iterations (it) remains almost bounded for the BDDC preconditioner, while it increases and is considerably higher for the unpreconditioned case (\emph{Schur}). 
	
	\begin{table}[!ht]
	    \centering
	    \caption{\emph{Optimality tests.} Optimality tests on $[0,5]$ ms, fixed time step $\tau=0.05$. Fixed number of $4\times4$ cells, each discretized with increasing number of finite elements. Reported number of linear iterations (it) and condition number $k_2$ at time $t=5$ ms. }
	    \label{tab: optimality matlab}
	    \begin{tabular}{cccccccc}
	        \toprule
	        \multirow{2}{*}{$Lcy$}  &\multirow{2}{*}{$H/h=6 \times Lcy$}  &\multicolumn{2}{c}{BDDC}    &\multicolumn{2}{c}{Schur}     &\multicolumn{2}{c}{CG}    \\
	        &   &$k_2$              &it                 &$k_2$              &it         &$k_2$              &it       \\   
	        \midrule
	        2   &12     &10.63      &19         &1.20e+03       &85         &2.26e+03       &118        \\
	        3   &18     &11.63      &20         &1.76e+03       &101        &4.31e+03       &160        \\
	        4   &24     &12.35      &21         &2.34e+03       &114        &7.09e+03       &204        \\
	        5   &30     &12.93      &22         &2.90e+03       &126        &1.04e+04       &246        \\
	        6   &36     &13.41      &23         &3.43e+03       &136        &1.47e+04       &288        \\
	        7   &42     &13.82      &24         &3.99e+03       &145        &1.93e+04       &326        \\
	        8   &48     &14.19      &24         &4.58e+03       &154        &2.47e+04       &365        \\
	        9   &54     &14.51      &24         &5.15e+03       &160        &3.11e+04       &405        \\
	        10  &60     &14.80      &25         &5.68e+03       &168        &3.79e+04       &443        \\
	        \bottomrule
	    \end{tabular}
	\end{table}
	
	\subsection{Dependence on the time step} \label{sec: dependence tau}
	In this Section, we investigate the dependence on the time step $\tau$.
	We fix the number of cells to $12\times12$, each cell discretized with $24\times4$ finite elements. The time interval considered is $[0,5]$ ms, and we vary the time step $\tau$ from $0.005$ to $0.1$. 
	It is very interesting to notice that the condition number ($k_2$) and iteration counts (it) of the \emph{Schur} and \emph{CG} cases in Table \ref{tab: dependence time step} increase slightly with shrinking time step. This behaviour, in contrast to standard parabolic problems, is caused by the algebraic equations in the system.
	The BDDC preconditioner does not suffer in terms of condition number or iteration counts, which increase only slightly with increasing time step size.
	
	\begin{table}[!ht]
	    \centering
	    \caption{\emph{Dependence on time step tests.} Testing the dependence on time step on $[0,5]$ ms, increasing time step from $\tau=0.005$ to $\tau=0.1$. Fixed number of $12\times12$ cells, each discretized with $24\times4$ finite elements. Reported number of linear iterations (it) and condition number $k_2$ at time $t=5$ ms.}
	    \label{tab: dependence time step}
	    \begin{tabular}{lcccccc}
	        \toprule
	        \multirow{2}{*}{$\tau$} &\multicolumn{2}{c}{BDDC}    &\multicolumn{2}{c}{Schur}     &\multicolumn{2}{c}{CG}    \\
	                   &$k_2$      &it            &$k_2$      &it         &$k_2$       &it       \\      
	        \midrule
	        $0.005$      &11.73     &25       &1.08e+05     &354       &2.54e+05   &630        \\
	        $0.01$       &11.93     &26       &9.94e+04     &375       &2.30e+05   &649       \\
	        $0.02$       &12.95     &26       &3.69e+04     &340       &9.38e+04   &584       \\
	        $0.05$       &15.40     &28       &2.05e+04     &305       &5.26e+04   &494        \\
	        $0.1$        &14.26     &28       &3.58e+04     &343       &8.77e+04   &573       \\
	        \bottomrule
	    \end{tabular}
	\end{table}

	\section{Conclusions} \label{sec: conclusions}
	We have designed, analyzed and tested Balancing Domain Decomposition by Constraints (BDDC) preconditioners for composite DG-type discretizations of cardiac cell-by-cell models, where each cell is associated with a subdomain of the domain decomposition. The keypoint in our construction is the definition of extended dual and primal spaces for the degrees of freedom, where additional constraints are introduced in order to transfer information between subdomains without affecting the overall discontinuity of the global solution.
	Our theoretical convergence rate analysis shows scalability and quasi-optimality of the proposed BDDC preconditioned cell-by-cell operator. We validated our theoretical bounds with extensive two-dimensional tests, where numerical results demonstrate the BDDC scalability, quasi-optimality, and independence from the time step size. 
	Future works will concern the parallelization of the cell-by-cell solver and its extension to three-dimensional models, as well as its integration in the Kaskade library \cite{gotschel2021flexible} with spectral deferred correction methods - in the same fashion as for the Bidomain, see \cite{chegini2022efficient, weiser_2022}. Further possible directions include the design of overlapping Schwarz preconditioners \cite{antonietti2007schwarz, antonietti2008multiplicative} and nonlinear solvers \cite{barnafi2022parallel} for implicit and linearly implicit time discretizations for DG discretizations of cell-by-cell models.

    \section*{Acknowledgements}
    This work was supported by the European High-Performance Computing Joint Undertaking EuroHPC under grant agreement No 955495 (MICROCARD) co-funded by the Horizon 2020 programme of the European Union (EU), the
    French National Research Agency ANR, the German Federal Ministry of Education and Research, the Italian ministry of economic development, the Swiss State Secretariat for Education, Research and Innovation, the
    Austrian Research Promotion Agency FFG, and the Research Council of Norway.



\begin{thebibliography}{99}
    \addcontentsline{toc}{section}{References}
	\bibliographystyle{siam}

	
	\bibitem{aliev1996simple} 
	R. R. Aliev and A. V. Panfilov,
	\emph{A simple two-variable model of cardiac excitation},
	Chaos, Solitons \& Fractals, 7.3 (1996), pp. 293-301.

    \bibitem{antonietti2007schwarz}
    P. F. Antonietti and B. Ayuso, 
    \emph{Schwarz domain decomposition preconditioners for discontinuous Galerkin approximations of elliptic problems: non-overlapping case},
    ESAIM: Math. Model. Numer. Anal., 41.1 (2007): 21-54.

    \bibitem{antonietti2008multiplicative}
    P. F. Antonietti and B. Ayuso,
    \emph{Multiplicative Schwarz methods for discontinuous Galerkin approximations of elliptic problems},
    ESAIM: Math. Model. Numer. Anal., 42.3 (2008): 443-469.
 
	\bibitem{bader1_2022} 
	F. Bader, M. Bendahmane, M. Saad, R. Talhouk,
    \emph{Microscopic Tridomain Model of Electrical Activity in the Heart with Dynamical Gap Junctions. Part 1 – Modeling and Well-Posedness},
    Acta Appl. Math. 179, 11 (2022).

	\bibitem{bader2_2022} 
	F. Bader, M. Bendahmane, M. Saad, R. Talhouk,
	\emph{Microscopic Tridomain Model of Electrical Activity in the Heart with Dynamical Gap Junctions. Part 2 – Derivation of the Macroscopic Tridomain Model by Unfolding Homogenization Method},
	Asymptotic Analysis, 132 (3-4), pp. 575--606 (2023).
	
	\bibitem{balay2022petsc} 
	S. Balay et al.,
	\emph{{PETS}c {W}eb page},
	\href{https://petsc.org/}{https://petsc.org/} (2022).

    \bibitem{barnafi2022parallel}
    N. A. Barnafi, L. F. Pavarino and S. Scacchi, 
    \emph{Parallel inexact Newton–Krylov and quasi-Newton solvers for nonlinear elasticity},
    Comp. Meth. Appl. Mech. Engrg., 400 (2022), 115557.
	

    \bibitem{canuto2014bddc}
    C. Canuto and L. F. Pavarino, A. B. Pieri,
    \emph{BDDC preconditioners for continuous and discontinuous Galerkin methods using spectral/hp elements with variable local polynomial degree}, 
    IMA J. Numer. Anal., 34(3) (2014), pp. 879--903.

    \bibitem{chegini2022efficient}
    F. Chegini, T. Steinke and M. Weiser,
    \emph{Efficient adaptivity for simulating cardiac electrophysiology with spectral deferred correction methods},
    under review (2022).
	
	\bibitem{cockburn2012discontinuous} 
	B. Cockburn, G. E. Karniadakis and C-W. Shu,
	\emph{Discontinuous {G}alerkin methods: theory, computation and applications},
	Vol. 11. Springer Science \& Business Media (2012).
	
	\bibitem{franzone2014mathematical}
	P. Colli Franzone, L.F. Pavarino and S. Scacchi,
	\emph{Mathematical cardiac electrophysiology},
	Springer, vol. 13 (2014).

    \bibitem{desouza2023boundary}
    G. Rosilho de Souza, R. Krause and S. Pezzuto,
    \emph{Boundary integral formulation of the cell-by-cell model of cardiac electrophysiology},
    arXiv preprint arXiv:2302.05281 (2023).

	\bibitem{dohrmann2003}
    C. Dohrmann,
    \emph{A preconditioner for substructuring based on constrained energy minimization},
    SIAM J. Sci. Comp. 25 (1),  (2003), pp. 246--258.
	
	\bibitem{dryja2013analysis} 
	M. Dryja, J. Galvis and M. Sarkis,
	\emph{A {FETI-DP} preconditioner for a composite finite element and discontinuous {G}alerkin method},
	SIAM J. Numer. Anal., 51-1 (2013), pp. 400--422.
	
	\bibitem{dryja2015deluxe} 
	M. Dryja, J. Galvis and M. Sarkis,
	\emph{A deluxe {FETI-DP} preconditioner for a composite finite element and {DG} method},
	Comput. Methods Appl. Math., 15-4 (2015), pp. 465-482.
	
	\bibitem{gotschel2021flexible} 
	S. Götschel, A. Schiela, and M. Weiser,
	\emph{Kaskade 7—{A} flexible finite element toolbox},
	Computers \& Mathematics with Applications 81  pp. 444-458 (2021).

    \bibitem{huynh2022parallel}
     N. M. M. Huynh, L. F. Pavarino, S. Scacchi,
    \emph{Parallel {N}ewton--{K}rylov {BDDC} and {FETI-DP} Deluxe Solvers for Implicit Time discretizations of the Cardiac {B}idomain Equations},
    SIAM J. Sci. Comp. 44(2): B224-B249 (2022).

    \bibitem{huynh2021newton}
     N. M. M. Huynh, 
    \emph{Newton-{K}rylov-{BDDC} deluxe solvers for non-symmetric fully implicit time discretizations of the {B}idomain model},
    Numer. Math. 152(4), pp. 841–879 (2022).

	\bibitem{tveito2021bis} 
	K. H. Jaeger, K. G. Hustad, X. Cai, A. Tveito, et al,
	\emph{Efficient Numerical Solution of the {EMI} Model Representing the Extracellular Space (E), Cell Membrane (M) and Intracellular Space (I) of a Collection of Cardiac Cells},
	Front. Physics 8, (2021).
	
	\bibitem{tveito2021tris}
	K. H. Jaeger, A. G. Edwards, W. R. Giles, A. Tveito. 
	\emph{From Millimeters to Micrometers; Re-introducing Myocytes in Models of Cardiac Electrophysiology}. Front. Physiol. 12:763584, (2021).

    \bibitem{jaeger2023reintroducing}
    K. H. Jæger and A. Tveito,
    \emph{Re-introducing the cell: the Extracellular-Membrane-Intracellular (EMI) model},
    in: Differential Equations for Studies in Computational Electrophysiology. Simula SpringerBriefs on Computing, vol 14. Springer, Cham (2023).

	
	
	
	\bibitem{pezzuto2022} 
	S. Pezzuto, G. Rosilho de Souza, R. Krause,
	\emph{Boundary integral discretization of the cell-to-cell bidomain model of cardiac electrophysiology},
	In:  8th European Congress on Computational Methods in Applied Sciences and Engineering (ECCOMAS), Oslo, Norway, June 2022.

    \bibitem{plank2007algebraic}
    G. Plank, M. Liebmann, R. Weber dos Santos, E. Vigmond, G. Haase, 
    \emph{Algebraic multigrid preconditioner for the cardiac bidomain model},
    IEEE transactions on biomedical engineering 54(4), (2007), pp. 585-96.
	
	\bibitem{potse_siampp2022} 
	M. Potse,
	\emph{Microscale cardiac electrophysiology on exascale supercomputers},
	In: SIAM Conference on Parallel Processing for Scientific Computing (PP22), Seattle, WA, USA, February 2022.
	
	\bibitem{rohr2004}	
	S. Rohr,
	\emph{Role of gap junctions in the propagation of the cardiac action potential},
	Cardiov. Res. 62(2), (2004), pp. 309--322.
	

	\bibitem{sathar2015}  
	S. Sathar, M.L. Trew, G. O’Grady,  L.K Cheng,
	\emph{A multiscale tridomain model for simulating bioelectric gastric pacing},
	IEEE Trans. Biomed. Eng. 62(11), (2015), pp. 2685 - 2692.
	

	\bibitem{toselli2006domain} 
	A. Toselli and O. B. Widlund,
	\emph{Domain decomposition methods},
	Algorithms and theory. Springer, 2005.
	
	\bibitem{tung1978} 	
	L. Tung,
	\emph{A bidomain model for describing ischemic myocardial d-c potentials},
	PhD thesis, M.I.T. Cambridge, Mass. (1978)
        
	\bibitem{tveito2017} 
	A. Tveito, et al.,
	\emph{A cell-based framework for numerical modeling of electrical conduction in cardiac tissue},
	Frontiers in Physics 48, (2017).
	
	\bibitem{tveito2021} 
	A. Tveito, K.-A. Mardal, M. E. Rognes,
	\emph{Modeling Excitable Tissue - The {EMI} Framework},
	Simula Springer Briefs in Computing 7, 2021.
	
	\bibitem{veneroni2006reaction}
    M. Veneroni,
    \emph{Reaction-diffusion systems for the microscopic cellular model of the cardiac electric field},
    Math. Meth. Appl. Sci., 29, (2006), pp. 1631--1661.
	
	\bibitem{veneroni2009reaction} 
	M. Veneroni,
	\emph{Reaction–diffusion systems for the macroscopic bidomain model of the cardiac electric field},
	Nonlinear Anal., Real World Appl. 10(2), (2009), pp. 849--868.

    \bibitem{vigmond2008solvers}
    E. J. Vigmond, R. Weber dos Santos, A. J. Prassl, M. Deo and G. Plank, 
    \emph{Solvers for the cardiac bidomain equations},
    Progress Biophys. Molec. Biol., 96 (1-3), (2008), pp. 3--18.

    \bibitem{weiser_2022} 
	M. Weiser and F. Chegini,
	\emph{Adaptive multirate integration of cardiac electrophysiology with spectral deferred correction methods}. In P. Nithiarasu and C. Vergara (eds.), CMBE22 -- 7th International Conference on Computational \& Mathematical Biomedical Engineering. (2022), pp. 528--531. 

    \bibitem{widlund1988iterative}
    O. Widlund,
    \emph{Iterative substructuring methods: algorithms and theory for elliptic problems in the plane}.
    In: First International Symposium on Domain Decomposition Methods for Partial Differential Equations, Philadelphia, PA. (1988).

    \bibitem{zampini2014dual}
    S. Zampini,
    \emph{Dual-primal methods for the cardiac bidomain model},
    Math. Mod. Meth. Appl. Sci. 24(4), (2014), pp. 667--696.

	
	
\end{thebibliography}
\end{document}